\documentclass[12pt]{amsart}
\usepackage{enumitem,amssymb,version,aliascnt,geometry,mathrsfs}
\usepackage[all]{xy}
\usepackage{hyperref}
\hypersetup{
pdfauthor={Olivier Haution},
pdftitle={Diagonalisable p-groups cannot fix exactly one point on projective varieties},
pdfkeywords={actions of p-groups, fixed points, degree formula},
hidelinks,
unicode
}

\swapnumbers

\DeclareFontFamily{U}{dutchcal}{\skewchar\font=45 }
\DeclareFontShape{U}{dutchcal}{m}{n}{<-> s*[1.0] dutchcal-r}{}
\DeclareFontShape{U}{dutchcal}{b}{n}{<-> s*[1.0] dutchcal-b}{}
\DeclareMathAlphabet{\mathlcal}{U}{dutchcal}{m}{n}
\SetMathAlphabet{\mathlcal}{bold}{U}{dutchcal}{b}{n}

\DeclareMathOperator{\id}{id}
\DeclareMathOperator{\Spec}{Spec}
\DeclareMathOperator{\im}{im}
\DeclareMathOperator{\CH}{CH}
\DeclareMathOperator{\Hom}{Hom}
\DeclareMathOperator{\Aut}{Aut}
\DeclareMathOperator{\Sym}{Sym}
\DeclareMathOperator{\Pic}{Pic}

\DeclareMathOperator{\Diag}{D}

\newcommand{\Oc}{\mathcal{O}}

\newcommand{\Lc}{\mathcal{L}}
\newcommand{\Nc}{\mathcal{N}}
\newcommand{\Fc}{\mathcal{F}}
\newcommand{\Ic}{\mathcal{I}}
\newcommand{\Rc}{\mathcal{R}}
\newcommand{\Wc}{\mathcal{W}}
\newcommand{\Vc}{\mathcal{V}}
\newcommand{\Zz}{\mathbb{Z}}
\newcommand{\Zo}{\mathcal{Z}}
\newcommand{\Fp}{\mathbb{F}_p}
\newcommand{\Pp}{\mathbb{P}}
\newcommand{\Gm}{\mathbb{G}_m}
\newcommand{\Au}{\mathbb{A}^1}
\newcommand{\Auc}{\mathbf{A}^1}
\newcommand{\mf}{\mathfrak{m}}
\newcommand{\pf}{\mathfrak{p}}
\newcommand{\Lchi}[1]{\mathscr{L}_{#1}}
\newcommand{\Lch}{\mathscr{L}}
\newcommand{\sch}{\mathlcal{s}}
\newcommand{\Char}[1]{\widehat{#1}}
\newcommand{\mup}{{\mu_p}}
\newcommand{\mun}{{\mu_n}}
\newcommand{\Gmp}{\Gm/p}
\newcommand{\Gmn}{\Gm/n}

\newtheorem*{theorem*}{Theorem}

\newtheorem{theorem}{Theorem}

\newaliascnt{proposition}{theorem}
\newtheorem{proposition}[proposition]{Proposition}
\aliascntresetthe{proposition}

\newaliascnt{lemma}{theorem}
\newtheorem{lemma}[lemma]{Lemma}
\aliascntresetthe{lemma}

\newaliascnt{corollary}{theorem}
\newtheorem{corollary}[corollary]{Corollary}
\aliascntresetthe{corollary}

\theoremstyle{definition}

\newaliascnt{remark}{theorem}
\newtheorem{remark}[remark]{Remark}
\aliascntresetthe{remark}

\newaliascnt{example}{theorem}
\newtheorem{example}[example]{Example}
\aliascntresetthe{example}

\newaliascnt{definition}{theorem}
\newtheorem{definition}[definition]{Definition}
\aliascntresetthe{definition}

\renewcommand{\theequation}{\thesection.\alph{equation}}
\newcommand{\rref}[1]{(\ref{#1})}
\newcommand{\dref}[2]{(\ref{#1}.\ref{#2})}
\newcommand{\tref}[3]{(\ref{#1}.\ref{#2}.\ref{#3})}

\begin{document}
\begin{abstract}
We prove an algebraic version of a classical theorem in topology, asserting that an abelian $p$-group action on a smooth projective variety of positive dimension cannot fix exactly one point. When the group has only two elements, we prove that the number of fixed points cannot be odd. The main tool is a construction originally used by Rost in the context of the degree formula. The framework of diagonalisable groups allows us to include the case of base fields of characteristic $p$.
\end{abstract}
\author{Olivier Haution}
\address{Mathematisches Institut, Ludwig-Maximilians-Universit\"at M\"unchen, Theresienstr.\ 39, D-80333 M\"unchen, Germany}
\title[Diagonalisable $p$-groups cannot fix exactly one point]{Diagonalisable $p$-groups cannot fix exactly one point on projective varieties}

\email{olivier.haution at gmail.com}
\thanks{This work was supported by the DFG grant HA 7702/1-1.}
\subjclass[2010]{14L30, 14C25, 14C17}
\keywords{actions of $p$-groups, fixed points, degree formula}
\date{\today}
\maketitle

\section*{Introduction}
The following result in algebraic topology was proved in the sixties:
\begin{theorem*}
An orientation preserving diffeomorphism of odd prime power order of a closed oriented positive dimensional manifold cannot fix exactly one point.
\end{theorem*}
This was initially a conjecture of Conner-Floyd \cite[\S45]{CF-book-1st}, proved first by Atiyah-Bott \cite[Theorem~7.1]{AB-FixedII} as a consequence of their fixed point formula, and later reproved by Conner-Floyd \cite[(8.3)]{CF-odd_period} using equivariant bordism. In the eighties, this result was generalised to actions of abelian $p$-groups (as opposed to cyclic $p$-groups) independently by Browder \cite[Corollary (1.6)]{Browder-PB} and Ewing-Stong \cite[\S3]{Ewing-Stong}.\\

In the present paper, we discuss the situation in algebraic geometry, and in particular prove the following statement.
\begin{theorem*}
Let $G$ be an abelian $p$-group acting on a smooth projective variety $X$ over an algebraically closed field of characteristic unequal to $p$. Assume that $X$ has no zero-dimensional component. Then $G$ cannot fix precisely one point.

If $G=\Zz/2$, the number of fixed points cannot be odd.
\end{theorem*}
Our main tool is a construction originally due to Rost (the case of $\Zz/2$ may be entirely treated using the related construction of \cite{df}, but we provide a self-contained proof). When the group $\Zz/p$ (with $p$ a prime number) acts on a variety $X$ over a field containing a root of unity of order $p$ (in particular of characteristic unequal to $p$), Rost defines in \cite{Rost-df} a cycle class $\varrho(X)$ in the modulo $p$ Chow group of the fixed locus. He proved the so-called degree formula \cite[Theorem 4.1]{Mer-df-notes} using this class in the case of the action by cyclic permutations of factors on the $p$-th power $X=Y^p$ of a variety $Y$, where the fixed locus is the diagonal $Y$. But this construction happens to be also useful in another extreme case, namely when the fixed locus is finite. In this situation, the degree of the cycle class $\varrho(X)$ is the number of fixed points (modulo $p$), counted with appropriate multiplicities. A crucial observation for the proof of the theorem stated above concerns an equivariance property of Rost's construction, when the $\Zz/p$-action extends to the action of a larger abelian group (we are however unable to construct a functorial lifting of the Rost operation to equivariant modulo $p$ Chow groups).

Let us emphasise the difference of context with the degree formula of \cite[Theorem 4.1]{Mer-df-notes}. The latter is used for questions of arithmetic nature (it asserts the existence of a zero-cycle of a certain degree), a typical consequence being the isotropy of a certain quadratic form. By contrast the theorem above is essentially geometric, and indeed most statements of this paper are not significantly weakened by assuming that the base field is algebraically closed.

Although the consideration of actions of ordinary groups --- as opposed to algebraic groups --- on varieties would suffice to prove the theorem stated above, we chose to write the paper using the more sophisticated framework of diagonalisable groups. This allows in particular to treat the case of infinitesimal diagonalisable $p$-groups (such as $\mup$ in characteristic $p$), which may be of some interest. More importantly, we found that most proofs become more transparent and considerably shorter when the language of diagonalisable groups is used (a notable exception is \S\ref{sect:equiv}, because the notion of equivariant cycles seems to be more complicated for non-constant groups). We prove the following more general theorem.
\begin{theorem*}
Let $G$ be a finite diagonalisable $p$-group acting on a projective variety $X$ over an arbitrary field. Assume that $X$ has no zero-dimensional connected component. Then the set underlying the fixed locus $X^G$ cannot be a single regular closed point of $X$ of degree prime to $p$.

If $G=\mu_2$, the set underlying $X^G$ cannot be an odd number of rational regular points.
\end{theorem*}

The structure of the paper is as follows. The purpose of the first section is mostly to provide some motivation to the reader. We state the main result in the case of the action of an ordinary $p$-group (postponing its proof until the end of the paper), and provide a series of examples aimed at testing the sharpness of the statement, and at motivating the consideration of diagonalisable groups in the rest of the paper (see e.g.\ \rref{ex:roots}). We state the more precise result for actions of the group $\Zz/2$, that is, for involutions, and provide counterexamples to its generalisation to actions of other $p$-groups.

After explaining our notation in \S\ref{sect:notation}, we recall in \S\ref{sect:actions} general facts concerning finite diagonalisable groups and their actions on algebraic varieties. We study in some details $\mun$-torsors, defining the objects $\Lch$ and $\sch$ which play an important role in the sequel.

We then explain in \S\ref{sect:operation} the construction of Rost's operation, already described in the preprints of Rost \cite{Rost-df} and Boisvert \cite{Boi-A-08}. Essentially the only new result in this section is the observation that everything works also in characteristic $p$, provided that $\Zz/p$ is replaced with $\mup$. Even though we are interested in an operation at the level of usual Chow groups, we have to consider higher $K$-cohomology groups in intermediate steps of the construction. Since there is essentially no additional cost, we construct the operation for an arbitrary $p$-torsion cycle module $M$; in the applications $M$ will be the modulo $p$ Milnor $K$-theory.

In \S\ref{sect:equiv}, we prove that if the $\mup$-action on a variety extends to an action of a larger finite diagonalisable group $G$, then the operation of \S\ref{sect:operation} produces a $G$-equivariant cycle. The results of \S\ref{sect:operation} and \S\ref{sect:equiv} are then used in \S\ref{sect:mainth} to prove the main theorem.

\section{Results for ordinary groups}
\numberwithin{theorem}{section}
\numberwithin{lemma}{section}
\numberwithin{proposition}{section}
\numberwithin{corollary}{section}
\numberwithin{example}{section}
\numberwithin{definition}{section}
\numberwithin{remark}{section}
\numberwithin{equation}{section}
\label{sect:results}

In order to state the main theorem, we will use the following terminology concerning group actions on varieties. A more general, but compatible, framework will be described in \S\ref{sect:alg_gp} and used from there on. Let $k$ be a field. An action of an ordinary group $G$ on a $k$-variety is a group morphism $G \to \Aut_k(X)$. For $g\in G$, the closed subscheme $X^g$ of $X$ is defined as the equaliser of the automorphism of $X$ induced by $g$ and the identity of $X$. The fixed locus $X^G$ is the closed subscheme of $X$ defined as the intersection of the subschemes $X^g$ for $g\in G$. Its set of $k$-points $X^G(k)$ coincides with the set of fixed points $X(k)^G$.

\begin{theorem}
\label{th:isolated:ordinary}
Let $X$ be a projective $k$-variety without connected component of dimension zero, and $G$ an ordinary abelian $p$-group acting on $X$. For every $g \in G$, assume that $k$ contains a root of unity whose order is the order of $g$. Then the set underlying $X^G$ cannot be a single regular closed point of $X$ of degree prime to $p$.
\end{theorem}
\begin{proof}
We may assume that $G$ is finite by \rref{lemm:finite} below. By \S\ref{sect:constant}, the $G$-action on $X$ corresponds to the action of some finite diagonalisable $p$-group. Thus the theorem is a special case of \rref{th:isolated}.
\end{proof}

\begin{lemma}
\label{lemm:finite}
Let $G$ be an ordinary group acting on a $k$-variety $X$. Then there is a finitely generated subgroup $G'$ of $G$ such that $X^G=X^{G'}$.
\end{lemma}
\begin{proof}
Assume the contrary. We construct a chain of finitely generated subgroups $G_i\subset G_{i+1}$ of $G$ such that $X^{G_{i+1}} \subsetneq X^{G_i}$. Since $X$ is a noetherian topological space, and the fixed loci are closed, we will obtain a contradiction and thus prove the lemma. Let $G_0=1$, and assume $G_i$ constructed for $i \geq 0$. By assumption $X^G \subsetneq X^{G_i}$. It follows that we may find $g\in G$ such that $X^g \cap X^{G_i} \subsetneq X^{G_i}$, and we let $G_{i+1} \subset G$ be the subgroup generated by $G_i$ and $g$.
\end{proof}

We now illustrate the necessity of the hypotheses of \rref{th:isolated:ordinary}.
\begin{example}[$G$ must be abelian]
Assume that $k$ is algebraically closed, and let $G$ be a non-abelian finite group of order prime to its exponential characteristic. Then there exists an irreducible $G$-representation $V$ over $k$ with $\dim V > 1$ (see e.g.\ \cite[\S I.3.1, Th\'eor\`eme 9]{Ser-78} for the case $k=\mathbb{C}$). The $G$-representation $V\oplus 1$, where $1$ denotes the trivial one-dimensional $G$-representation, admits a single one-dimensional subrepresentation. Therefore the group $G$ acts on the $k$-variety $\Pp(V \oplus 1)$ with a single fixed point.
\end{example}

\begin{example}[the characteristic of $k$ cannot be $p$]
Assume that the characteristic of $k$ is $2$, and consider the action of the group $G=\Zz/2$ on $X= \Pp^1_k$ given by the involution $[x:y] \mapsto [y:x]$. Then the set underlying $X^G$ is the single rational point $[1:1]$. Note however that the scheme $X^G$ does not coincide with this closed point, in fact $X^G \simeq \Spec k[t]/(t^2)$.
\end{example}

\begin{example}[$k$ must contain enough roots of unity]
\label{ex:roots}
Assume that the characteristic of $k$ is not $p$. The endomorphism $[x_0: \cdots :x_{p-1}] \mapsto [x_1: \cdots :x_{p-1}:x_0]$ of $\Pp^{p-1}_k$ has order $p$, hence induces an action of the group $G=\Zz/p$ on $\Pp^{p-1}_k$. The hyperplane $X$ defined by $x_0 +\cdots + x_{p-1} =0$ is $G$-invariant, and has no zero-dimensional component when $p$ is odd. Over an algebraic closure $\overline{k}$ of $k$, the fixed locus $(X^G)_{\overline{k}} = (X_{\overline{k}})^G$ consists of the $(p-1)$ points $[1:\xi:\xi^2:\cdots:\xi^{p-1}]$, where $\xi$ runs over the primitive $p$-th roots of unity in $\overline{k}$. In particular that $X^G$ is a single closed point of degree $p-1$ when $k=\mathbb{Q}$.
\end{example}

Example \rref{ex:roots} also shows that $X^G$ may consist of $n$ regular points with $n$ prime to $p$ (namely $n=p-1$) when $k$ is algebraically closed and $G=\Zz/p$ with $p \neq 2$. This contrasts with the situation for $p=2$:
\begin{theorem}
\label{th:involutions}
Assume that $k$ is algebraically closed of characteristic unequal to two. Let $X$ be a projective $k$-variety without connected component of dimension zero. Then no $k$-involution of $X$ may fix precisely an odd number of points, all of which are regular.
\end{theorem}
\begin{proof}
An involution is the same thing as an action of the group $\mathbb{Z}/2$, and thus of $\mu_2$ since $k$ contains a non-trivial square root of the unity (see \S\ref{sect:constant}). Therefore the statement follows from \rref{th:odd}. Alternatively, it follows directly from \cite[Proposition 4.3 and Theorem 4.6]{df}.
\end{proof}

\begin{example}[\rref{th:involutions} does not generalise to other $2$-groups]
Assume that the characteristic of $k$ is not two. The commuting involutions
\[
[x :y:z] \mapsto [-x :y:z] \quad \text{ and } \quad [x :y:z] \mapsto [x :-y:z] 
\]
define an action of $G=\mathbb{Z}/2 \times \mathbb{Z}/2$ on $X=\mathbb{P}^2_k$. Then the set underlying $X^G$ consists of the three points $[1:0:0],[0:1:0],[0:0:1]$.

For an example with $G$ cyclic and the same $X$, assume that $k$ contains a primitive fourth root of unity $i$. The automorphism of order four
\[
[x :y:z] \mapsto [-x :iy:z]
\]
induces an action of $G=\mathbb{Z}/4$ on $X$. The set underlying $X^G$ again consists of the three points $[1:0:0],[0:1:0],[0:0:1]$.
\end{example}

\section{Notation and basic facts}
\numberwithin{theorem}{subsection}
\numberwithin{lemma}{subsection}
\numberwithin{proposition}{subsection}
\numberwithin{corollary}{subsection}
\numberwithin{example}{subsection}
\numberwithin{definition}{subsection}
\numberwithin{remark}{subsection}
\numberwithin{equation}{subsection}
\renewcommand{\theequation}{\thesubsection.\alph{equation}}

\label{sect:notation}

\subsection{Varieties} The letter $k$ will denote a base field. A variety, or $k$-variety, will be a quasi-projective scheme over $k$. The residue field at a point $x$ of a $k$-variety will be denoted by $k(x)$. A point $x$ of a variety $X$ will be called regular if the local ring $\Oc_{X,x}$ is regular, and we will say that a variety is regular when all its points are regular. A morphism of varieties is a morphism of $k$-schemes. A flat morphism of varieties will always mean a flat morphism with a relative dimension. When $X$ is covered by open subschemes $U_i$ for $i\in I$, the words ``replacing $X$ with the cover $\{U_i,i\in I\}$'' will mean successively replacing $X$ with each $U_i$.

\subsection{Algebraic groups}
\label{sect:alg_gp}
We refer e.g. to \cite[I]{SGA3-1} for the basic definitions concerning algebraic groups and their actions on varieties. An algebraic group will mean a group scheme of finite type over $k$, and actions will always be over $k$.

If an algebraic group $G$ acts on a variety $X$, an open or closed subscheme $Y$ of $X$ will be called $G$-invariant if the restriction of the action morphism $G \times Y \to X$ factors through $Y$. The scheme-theoretic closure of a $G$-invariant open subscheme is a $G$-invariant closed subscheme. 

When an algebraic group $G$ is finite, we denote by $|G|$ its order, defined as the dimension of the $k$-vector space $H^0(G,\Oc_G)$. A finite algebraic group will be called a $p$-group if $p$ is a prime number and the order of $G$ is a power of $p$.

\subsection{Algebraic cycles}
We will use the notation of \cite{Ful-In-98} concerning algebraic cycles, with the following variations. The group of cycles on a variety $X$, denoted by $\Zo(X)$ (instead of $Z_*(X)$), is the free abelian group generated by the classes $[Z]$, for $Z$ an integral closed subscheme of $X$. The Chow group of $X$, denoted by $\CH(X)$ (instead of $A_*(X)$), is the quotient of $\Zo(X)$ modulo the relation of rational equivalence.

\subsection{The degree of a morphism}
\label{def:deg}
We say that a morphism $f\colon Y \to X$ has degree $m$ if $f_*[Y] = m \cdot [X] \in \Zo(X)$. We will also write $m = \deg f$. The morphism $f$ always has a degree when $X$ is integral and $Y$ has no irreducible component of dimension $< \dim X$. 

Setting $\deg [Z] = \deg (Z \to \Spec k)$ for $Z$ an integral closed subscheme of $X$ induces a morphism $\deg \colon \Zo(X) \to \Zz$, which descends to $\CH(X) \to \Zz$ when the variety $X$ is projective. When $p$ is a prime, we will also write $\deg$ for the morphism $\CH(X)/p \to \Fp$.

\subsection{The sheaf of invertible functions}
We denote by $\Gm$ the Zariski sheaf of groups such that $H^0(X,\Gm) = H^0(X,\Oc_X)^\times$ for any variety $X$. If $t\in H^0(X,\Gm)$ and $f\colon Y \to X$ is a morphism, we will often write $t\in H^0(Y,\Gm)$ instead of $f^*(t)$.

\subsection{The deformation variety}\cite[(10.4)]{Rost-Chow}
\label{sect:deformation}
Let $Y\to X$ be a closed immersion, and $\Ic$ the corresponding quasi-coherent ideal of $\Oc_X$. The deformation variety $D$ is the scheme affine over $X$ defined by the quasi-coherent $\Oc_X$-algebra
\[
\Rc = \bigoplus_{n \in \Zz} \Ic^{-n} t^n \subset \Oc_X[t,t^{-1}]
\]
where $\Ic^n = \Oc_X$ for $n \leq 0$. The global section $t \in H^0(D,\Oc_D)$ induces a flat morphism $D \to \Au$, and we have a commutative diagram with cartesian squares
\[ \xymatrix{
X \ar[r] & X \times \Au & X \times (\Au-0) \ar[l]\\
N\ar[r] \ar[u] & D \ar[u] & X \times (\Au-0) \ar[l] \ar[u]_= \\ 
Y \ar[r] \ar[u]& Y \times \Au \ar[u]& \ar[l] Y \times (\Au -0)\ar[u]
}\]
where $N$ is the normal cone of the immersion, and left horizontal arrows are closed immersions, and right horizontal ones their open complements.

\subsection{Cycle modules}
A cycle module will mean a cycle module over $k$, in the sense of \cite{Rost-Chow}. When $M$ is a cycle module and $X$ a variety, the complex of cycles on $X$ with coefficients in $M$ will be denoted by $C(X,M)$, and the Chow group with coefficients in $M$ (the direct sum of the homology groups of $C(X,M)$) by $A(X,M)$. The differential will be denoted by $d \colon C(X,M) \to C(X,M)$. When $E$ is a (closed or open) subscheme of $X$ we denote by $x \mapsto x|_E$ the projection $C(X,M) \to C(E,M)$. An element $t \in H^0(X,\Gm)$ induces an endomorphism $\{t\}$ of $C(X,M)$ and of $A(X,M)$, see \cite[(3.6), (4.6.3)]{Rost-Chow}.

We will use the following statement which does not appear explicitly in \cite{Rost-Chow}.
\begin{lemma}
\label{lemm:connecting}
Let $Y\to X$ and $i\colon Z \to Y$ be two closed immersions. Denote by $u\colon X-Y \to X-Z$ the open immersion. Then for any cycle module $M$ the following diagram commutes (horizontal arrows are connecting homomorphisms).
\[ \xymatrix{
A(X-Z,M)\ar[rr] \ar[d]_{u^*} && A(Z,M) \ar[d]^{i_*} \\ 
A(X-Y,M) \ar[rr] && A(Y,M)
}\]
\end{lemma}
\begin{proof}
When $B$ is a locally closed subscheme of $X$, we will view $C(B,M)$ as a subgroup of $C(X,M)$. Any element of $A(X-Z,M)$ is represented by some $\alpha \in C(X-Z,M) \subset C(X,M)$ such that $d(\alpha)|_{X-Z}=0$. Writing $\beta = \alpha|_{X-Y}$ and $\gamma = \alpha|_{Y-Z}$, we have $\alpha = \beta + \gamma$. Now we compute in $C(Y,M)$:
\begin{align*}
\delta_Y \circ u^*(\alpha)
&=d(\beta)|_Y\\
&= d(\beta)|_Z + d(\beta)|_{Y-Z} \\ 
&= d(\alpha)|_Z - d(\gamma)|_Z + d(\beta)|_{Y-Z}\\
&= d(\alpha)|_Z - d(\gamma)|_Z - d(\gamma)|_{Y-Z} && \text{ since $d(\alpha) \in C(Z,M)$}\\
&= d(\alpha)|_Z - d(\gamma)|_Y\\
&= i_* \circ \delta_Z(\alpha) - d(\gamma)|_Y.
\end{align*}
Since $\gamma \in C(Y-Z,M) \subset C(Y,M)$ and $Y$ is closed in $X$, the element $d(\gamma)|_Y$ belongs to the image of the differential of $C(Y,M)$, hence vanishes in $A(Y,M)$.
\end{proof}

\subsection{The zero-scheme of a section}
\label{sect:D}
Let $\Fc$ be a quasi-coherent $\Oc_X$-module. A section $s \in H^0(X,\Fc)$ may be viewed as a morphism $s \colon \Oc_X \to \Fc$. The image of the dual morphism $s^\vee \colon \Fc^\vee \to \Oc_X$ is a quasi-coherent ideal $\Ic(s)$ of $\Oc_X$, and we denote by $Z_{\Fc}(s)$ the corresponding closed subscheme of $X$. We say that the section $s$ is \emph{nowhere vanishing} if $Z_{\Fc}(s) = \varnothing$. We denote by $D_{\Fc}(s)$ the open complement of $Z_{\Fc}(s)$ in $X$. Observe that, for any $s,t \in H^0(X,\Fc)$ we have
\begin{equation}
\label{eq:D}
D_{\Fc}(s + t) \subset D_{\Fc}(s) \cup D_{\Fc}(t).
\end{equation}

If $\Lc$ is an invertible module and $s \in H^0(X,\Lc)$, then $D_{\Lc}(s)$ is the locus where the morphism $s \colon \Oc_X \to \Lc$ is an isomorphism. Thus the datum of a nowhere vanishing section of $\Lc$ is equivalent to that of a trivialisation of $\Lc$. 

If $X=\Spec A$ and $s \in A$ is viewed as a section of $\Oc_X$, then we will write $Z(s) = \Spec A/sA$ for $Z_{\Oc_X}(s)$ and $D(s) = \Spec A[s^{-1}]$ for $D_{\Oc_X}(s)$.

\subsection{The first Chern class}
Let $M$ be a cycle module. Let $q \colon L \to X$ be a line bundle with zero-section $i \colon X \to L$ and $\Oc_X$-module of sections $\Lc$. It first Chern class is
\[
c_1(\Lc)=c_1(L)=(q^*)^{-1} \circ i_* \colon A(X,M) \to A(X,M),
\]
and we will also consider the operator
\[
c(-\Lc) = \sum_{j = 0}^{\dim X} (-c_1(\Lc))^j \colon A(X,M) \to A(X,M).
\]

\begin{lemma}
\label{lemm:c1}
Let $L$ be a line bundle on a variety $X$, and $M$ a cycle module. 

\begin{enumerate}[label=(\roman*),ref=\roman*]
\item \label{lemm:c1:Gm} Let $t \in H^0(X,\Gm)$. Then $c_1(L) \circ \{t\} = \{t\} \circ c_1(L)$.

\item \label{lemm:c1:proper} Let $f\colon Y \to X$ be a proper morphism. Then  $f_* \circ c_1(f^*L) = c_1(L) \circ f_*$.

\item \label{lemm:c1:flat} Let $f\colon Y \to X$ be a flat morphism. Then $f^* \circ c_1(L) = c_1(f^*L) \circ f^*$.

\item \label{lemm:c1:conn} Let $Y$ be a closed subscheme of $X$, and $\delta \colon A(X-Y,M) \to A(Y,M)$ the connecting homomorphism. Then $\delta \circ c_1(L|_{X-Y}) = c_1(L|_Y) \circ \delta$.
\end{enumerate}
\end{lemma}
\begin{proof}
This follows from \cite[(4.1), (4.2.1), (4.3.1), (4.4)]{Rost-Chow}.
\end{proof}

The following direct construction of the first Chern class will be used in \S\ref{sect:equiv}.
\begin{lemma}
\label{lemm:c1_merom}
Let $X$ be a variety and $\Lc$ an invertible $\Oc_X$-module. Let $U$ be a dense open subscheme of $X$ and $l \in H^0(U,\Lc)$ a nowhere vanishing section. Let $U_i$ be a covering of $X$ by open subschemes, and $l_i \in H^0(U_i,\Lc)$ nowhere vanishing sections. Let $s_i \in H^0(U \cap U_i, \Gm)$ be such that $l_i s_i = l$ on $U \cap U_i$. Then the cycle class $c_1(\Lc)[X] \in \CH(X)$ is represented by a cycle in $\Zo(X)$ whose restriction to each $U_i$ is the image of $[U\cap U_i]$ under the composite (we denote by $u_i \colon U \cap U_i \to U_i$ the open immersion)
\[
\Zo(U\cap U_i) \xrightarrow{\{s_i\}} C(U\cap U_i,K_1) \xrightarrow{(u_i)_*} C(U_i,K_1) \xrightarrow{d} \Zo(U_i).
\]
\end{lemma}
\begin{proof}
We may assume that $X$ is integral. Then for each $i$ such that $U_i \neq \varnothing$, we may view $s_i$ as an invertible rational function on $X$, and the family $(U_i,s_i)$ defines a Cartier divisor $S$ on $X$ such that $\Oc(S) \simeq \Lc$. By \cite[Example 3.3.2, Theorem 3.2(f)]{Ful-In-98}, the Weil divisor associated with $S$ is a cycle in $\Zo(X)$ satisfying the required conditions.
\end{proof}

\subsection{The sheaf \texorpdfstring{$\Gmn$}{Gm/n}}
The letter $n$ will denote an integer $\geq 1$.
\begin{definition}
\label{def:Gmn}
We denote by $\Gmn$ the cokernel of the $n$-th power map $\Gm \to \Gm$ in the category of Zariski sheaves.
\end{definition}

\begin{remark}
\label{rem:sheafification}
The sheaf $\Gmn$ is the Zariski sheafification of the presheaf $U \mapsto H^1_{fppf}(U,\mu_n)$.
\end{remark}

Let $X$ be a variety and $M$ a cycle module such that $n \cdot M=0$. Let $t \in H^0(X,\Gmn)$. The stalk of $t$ at a point $x \in X$ is an element of $(\Oc_{X,x})^\times/(\Oc_{X,x})^{\times n}$, and its image in $k(x)^\times/k(x)^{\times n}$ acts on the left $k(x)^\times$-module $M(k(x))$ (which is assumed to be $n$-torsion). Thus we obtain a group morphism
\[
\{t\} \colon C(X,M) \to C(X,M).
\]

Note that any point of $X$ is contained in an open subscheme $U$ such that (the restriction of) $t$ lifts to a section $t' \in H^0(U,\Gm)$. Then the endomorphisms $\{t\}$ and $\{t'\}$ of $C(U,M)$ coincide.

\begin{lemma}
\label{lemm:Gmn}
Let $X$ be a variety and $M$ a cycle module such that $n \cdot M=0$. Let $t \in H^0(X,\Gmn)$. Then (at the cycle level $C(-,M)$):
\begin{enumerate}[label=(\roman*),ref=\roman*]
\item \label{lemm:Gmn:Gm} For any $u\in H^0(X,\Gm)$, we have $\{u\} \circ \{t\} = - \{t\} \circ \{u\}$.
\item \label{lemm:Gmn:pushforward} If $f \colon Y \to X$ is a morphism, then $f_*\circ \{f^*t\} = \{t\} \circ f_*$.
\item \label{lemm:Gmn:flat} If $f \colon Y \to X$ is a flat morphism, then $f^* \circ \{t\} = \{f^*t\} \circ f^*$.
\item \label{lemm:Gmn:d} $d \circ \{t\} = -\{t\} \circ d$.
\end{enumerate}
\end{lemma}
\begin{proof}
Each statement may be proved after replacing $X$ with an open cover, so that we may assume that $t$ lifts to $H^0(X,\Gm)$. Thus the statements follow respectively from the graded-commutativity of Milnor's $K$-theory, and \cite[(4.2.1), (4.3.1), (4.6.3)]{Rost-Chow}.
\end{proof}

\subsection{Gradings}
We now state our terminology concerning graded algebras and modules. These notions extend to sheaves of algebras and modules over them.

Let $H$ be an (ordinary) abelian group and $A$ a commutative ring with unity. An $H$-grading on a commutative $A$-algebra $R$ is a decomposition of the $A$-module $R=\bigoplus_{h \in H} R_h$ such that $R_h \cdot R_i \subset R_{h+i}$ for all $h,i \in H$. An $H$-grading on an $R$-module $M$ is a decomposition of the $A$-module $M=\bigoplus_{h \in H} M_h$ such that $R_h \cdot M_i \subset M_{h+i}$ for all $h,i \in H$. Elements of $\bigcup_{h \in H} M_h$ are called homogeneous. A morphism of $H$-graded modules is a morphism compatible with the gradings. 

A graded submodule of an $H$-graded module $M$ is a submodule $N \subset M$ which is generated by homogeneous elements. It is equivalent to require that for every $n \in N$ and $h\in H$ the component of $n$ in $M_h$ belong to $N$. In this case, the modules $N$ and $M/N$ inherit $H$-gradings such that $N_h = N \cap M_h$ and $(M/N)_h = M_h /N_h$.

A graded ideal of an $H$-graded $A$-algebra $R$ is a graded submodule of $R$. If $I,J$ are graded ideals, then the ideal $IJ$ is graded.

\section{Actions of finite diagonalisable groups on varieties}
\label{sect:actions}

\subsection{Quotients by finite algebraic groups}
\label{sect:quotients}
We gather some results on finite algebraic group actions which will repeatedly be used in the paper without explicit reference.

\begin{proposition}
\label{prop:open_action}
Let $G$ be a finite algebraic group acting on a variety $X$.
\begin{enumerate}[label=(\roman*),ref=\roman*]
\item \label{prop:open_action:cover} The variety $X$ is covered by affine $G$-invariant open subschemes.

\item \label{prop:open_action:sat}
Any open subscheme $U$ of $X$ contains a unique maximal $G$-invariant open subscheme $U'$. In addition:
\begin{enumerate}[label=(\alph*),ref=\alph*]
\item \label{prop:open_action:sat:closed} If $U$ is closed in $X$, the same is true for $U'$.
\item \label{prop:open_action:sat:inv} Any $G$-invariant closed subscheme of $X$ contained in $U$ is contained in $U'$.
\end{enumerate}
\end{enumerate}
\end{proposition}
\begin{proof}
\eqref{prop:open_action:cover}: See \cite[V, \S5]{SGA3-1}, the main point is that our varieties are quasi-projective, so that any finite subset of $X$ is contained in an affine open subscheme. 

\eqref{prop:open_action:sat}: Denote by $p \colon G\times X \to X$ the second projection, and by $a \colon G \times X \to X$ the action morphism. Consider the automorphism $\varepsilon \colon G \times X \to G \times X$ given by $(g,x) \mapsto (g^{-1},g \cdot x)$.  Since $p$ is open (being flat of finite type) and finite, the same is true for $a = p \circ \varepsilon$. Consider the sets $F = X-U$ and $S=a(p^{-1}F)$. Then $F=a(1_G \times F) \subset a(p^{-1}F) = S$. Denoting by $\mu_G\colon G \times G \to G$ the group operation, we have
\[
a (p^{-1}S) = a \circ (\id_G \times a) (G \times G \times F) = a \circ (\mu_G \times \id_F)(G \times G \times F) = a(G \times F) = S.
\]
This proves that $U'=X-S$ is $G$-invariant. If $F$ is open in $X$, then so is $S$ (because $a$ is open), proving \eqref{prop:open_action:sat:closed}. If $Z$ is a $G$-invariant open (resp.\ closed) subscheme of $X$ contained in $U$, then $Z \cap F =\varnothing$, hence $a^{-1}Z=p^{-1}Z$ does not meet $p^{-1}F$ in $G \times X$, so that $S \cap Z = \varnothing$, and thus $Z \subset U'$. The maximality and unicity of $U'$ follow (resp.\ \eqref{prop:open_action:sat:inv} follows).
\end{proof}

\begin{proposition}
\label{prop:quotient}
Let $G$ be a finite algebraic group acting on a variety $X$. Then there is a categorical quotient $\varphi=\varphi_X \colon X \to X/G$ in the category of $k$-varieties. In addition:

\begin{enumerate}[label=(\roman*),ref=\roman*]
\item \label{prop:quotient:finite-surj} The morphism $\varphi$ is finite and surjective.

\item \label{prop:quotient:open} If $U$ is a $G$-invariant open subscheme of $X$, then $U/G$ is an open subscheme of $X/G$, and $U=\varphi_X^{-1}(U/G)$.
\end{enumerate}
\end{proposition}
\begin{proof}
See \cite[V, \S5]{SGA3-1} for the existence of the quotient, and for \eqref{prop:quotient:finite-surj}. Assertion \eqref{prop:quotient:open} follows from \cite[V, Lemme 1.1]{SGA3-1}.
\end{proof}

When $f\colon Y \to X$ is a $G$-equivariant morphism, we will denote by $f/G \colon Y/G \to X/G$ the induced morphism.

Let $S' \to S$ be a morphism of varieties with a trivial $G$-action, and $X \to S$ a $G$-equivariant morphism. Categorical considerations show that the varieties $(X \times_S S')/G$ and $(X/G) \times_S S'$ are naturally isomorphic.

\subsection{Finite diagonalisable groups}(see \cite[\S2.2]{Waterhouse} or \cite[I, \S4.4]{SGA3-1})
Let $\Gamma$ be a finite abelian (ordinary) group. The functor associating to each commutative $k$-algebra $R$ the group of group morphisms $\Gamma \to R^{\times}$ is represented by a finite commutative algebraic group $\Diag(\Gamma)$. Algebraic groups of this type are called \emph{finite diagonalisable}. The coordinate ring of $\Diag(\Gamma)$ is the group algebra $k[\Gamma]$ over $k$, defined as the $k$-vector space on the basis $e_g$ for $g\in \Gamma$, with unit $e_0$, multiplication given by $e_g \cdot e_h = e_{g+h}$, comultiplication by $e_g \mapsto e_g \otimes e_g$, coinverse by $e_g \mapsto e_{-g}$ and counit by $e_g \mapsto 1$.

We denote by $\Char{G}$ the character group of an algebraic group $G$, defined as the (ordinary) abelian group of group-like elements in the Hopf algebra $H^0(G,\Oc_G)$. The associations $G\mapsto \Char{G}$ and $\Gamma \mapsto \Diag(\Gamma)$ induce an anti-equivalence between the category of finite abelian groups and the category of finite diagonalisable algebraic groups. For every integer $n\geq 1$, we let $\mun=\Diag(\Zz/n)$ be the finite diagonalisable group such that $\Char{\mun}=\Zz/n$.

In the affine case, actions of finite diagonalisable groups correspond to gradings of the coordinate ring:
\begin{lemma}[{\cite[I, 4.7.3.1]{SGA3-1}}]
\label{lemm:action_grading}
Let $\psi \colon X \to S$ be an affine morphism of varieties and $G$ a finite diagonalisable group acting trivially on $S$. The datum of a $G$-action on $X$ such that $\psi$ is $G$-equivariant is equivalent to that of a $\Char{G}$-grading on the $\Oc_S$-algebra $\psi_*\Oc_X$. 
\end{lemma}

\begin{proposition}
\label{prop:quotient_graded}
Let $G$ be a finite diagonalisable group acting on a variety $X$. Then the morphism $\Oc_{X/G} \to \varphi_* \Oc_X$ induces an isomorphism $\Oc_{X/G} \simeq (\varphi_* \Oc_X)_0$.
\end{proposition}
\begin{proof}
By {\cite[I, \S4.7.3]{SGA3-1}}, the action morphism $G \times X \to X$ corresponds to the morphism of $\Oc_{X/G}$-algebras $\varphi_*\Oc_X \to \varphi_*\Oc_X \otimes_k k[\Char{G}]$ given by $\sum_{g\in\Char{G}} \pi_g \otimes e_g$, where $\pi_g \colon \varphi_*\Oc_X \to \varphi_*\Oc_X$ is the projection onto the $g$-th component. Thus \cite[V, 1 b) and Th\'eor\`eme 4.1 (i)]{SGA3-1} implies that $\Oc_{X/G} \to \varphi_*\Oc_X$ is the equaliser in the category of sheaves of rings on $X/G$ of the two morphisms $\varphi_*\Oc_X \to \varphi_*\Oc_X \otimes_k k[\Char{G}]$ given by $\id \otimes e_0$ and $\sum_{g \in \Char{G}} \pi_g \otimes e_g$. The statement follows.
\end{proof}

\begin{lemma}
\label{lemm:quotient-map}
Let $G$ be a finite diagonalisable group and $f\colon Y \to X$ a $G$-equivariant morphism. If $f$ is respectively
\begin{enumerate}[label=(\roman*),ref=\roman*]
\item \label{lemm:quotient-map:proper} proper,
\item \label{lemm:quotient-map:closed} a closed immersion,
\end{enumerate}
then the same is true for $f/G$.
\end{lemma}
\begin{proof}
\eqref{lemm:quotient-map:proper}: The morphism $f/G$ is proper because $\varphi_X$ is proper and $\varphi_Y$ is surjective \cite[(5.4.2.ii) and (5.4.3.ii)]{ega-2}.

\eqref{lemm:quotient-map:closed}: We may assume that $f$ is given by a surjective morphism of $\Char{G}$-graded $k$-algebras $A \to B$. Then $A_0 \to B_0$ is surjective.
\end{proof}

\subsection{Equivariant modules}
Let $G$ be a finite diagonalisable group acting on a variety $X$. A \emph{$G$-equivariant structure} on an $\Oc_X$-module $\Fc$ is a $\Char{G}$-grading on the $\varphi_*\Oc_X$-module $\varphi_*\Fc$. A $G$-representation over $k$ (the case $X=\Spec k$) is a $k$-vector space $\Vc$ together with a $k$-linear decomposition $\Vc = \bigoplus_{g \in \Char{G}} \Vc_g$. The dual representation $\Vc^\vee$ is defined by $(\Vc^\vee)_g = (\Vc_{-g})^\vee$ for every $g \in \Char{G}$.

If $f\colon Y \to X$ is a $G$-equivariant morphism and $\Fc$ a quasi-coherent $G$-equivariant $\Oc_X$-module, the morphism of $\Oc_{Y/G}$-modules ${\varphi_Y}_*\circ f^*\Fc \to (f/G)^* \circ {\varphi_X}_*\Fc$ is an isomorphism \cite[(1.5.2)]{ega-2}. This induces a $G$-equivariant structure on the $\Oc_Y$-module $f^*\Fc$.

\subsection{Homogeneous sections}
\label{def:weight}
Let $G$ be a finite diagonalisable group acting on a variety $X$. Let $\Fc$ be a quasi-coherent $G$-equivariant $\Oc_X$-module and $g \in \Char{G}$. We say that a section in $H^0(X,\Fc)$ \emph{has weight $g$} if it belongs to the subgroup
\[
H^0(X/G,(\varphi_*\Fc)_g) \subset H^0(X/G,\varphi_*\Fc) = H^0(X,\Fc).
\]
A section of weight $g$ for some $g\in \Char{G}$ will be called \emph{homogeneous} (or \emph{$\Char{G}$-homogeneous}).

If $s\in H^0(X,\Fc)$ is homogeneous, then $\varphi_*\Ic(s) = \im (\varphi_*(s^\vee) \colon \varphi_*(\Fc^\vee) \to \varphi_*\Oc_X)$ (see \S\ref{sect:D}) is a $\Char{G}$-graded ideal of $\varphi_*\Oc_X$. It follows that the closed, resp.\ open, subscheme $Z_{\Fc}(s)$, resp.\ $D_{\Fc}(s)$, of $X$ is $G$-invariant.

\subsection{The fixed locus}

\begin{definition}
Let $G$ be a finite diagonalisable group acting on a variety $X$. Since the quotient morphism $\varphi\colon X \to X/G$ is affine, there is a unique coherent ideal $\Ic$ of $\Oc_X$ such that the ideal $\varphi_*\Ic$ of $\varphi_*\Oc_X$ is generated by $(\varphi_*\Oc_X)_g$ for $g \in \Char{G}-\{0\}$. We let $X^G$ be the corresponding closed subscheme of $X$. The functor associating to a variety $T$ with trivial $G$-action the set of $G$-equivariant morphisms $T \to X$ is represented by $X^G$.
\end{definition}

Note that $X^G$ is a $G$-invariant closed subscheme of $X$ with trivial $G$-action, and that if $Z$ is a $G$-invariant closed or open subscheme of $X$, then $Z^G = X^G \cap Z$.

A $G$-equivariant morphism $f\colon Y \to X$ induces a morphism $f^G \colon Y^G \to X^G$. When $Y$ and $X$ are two varieties with a $G$-action, we will endow the product $X \times Y$ (over $k$) with the diagonal $G$-action; then $(X \times Y)^G= X^G \times Y^G$.

\begin{lemma}
\label{lemm:fixed_reg}
Let $G$ be a finite diagonalisable group acting on a regular variety $X$. Then the variety $X^G$ is regular.
\end{lemma}
\begin{proof}
Let $x$ be a point of $X^G$ and $R=\Oc_{X,x}$. Any open neighborhood of $x$ contains a $G$-invariant open neighborhood by \tref{prop:open_action}{prop:open_action:sat}{prop:open_action:sat:inv}, which in addition may be taken to be affine by \dref{prop:open_action}{prop:open_action:cover}. Therefore the $k$-algebra $R$ is a direct limit of $\Char{G}$-graded $k$-algebras, and is thus naturally $\Char{G}$-graded \cite[II, \S11, N$^\circ$3, Remarque 3)]{Bou-A-13}. In addition it follows from \cite[II, \S6, N$^\circ$2, Propositions 3 and 4]{Bou-A-13} that the ideal $I$ of $R$ generated by $R_g$ for $g \in \Char{G}-\{0\}$ is the kernel of the surjection $R \to \Oc_{X^G,x}$. It will suffice to prove that the noetherian local ring $R/I$ is regular.

The maximal ideal $\mf$ of $R$ is $\Char{G}$-graded, as is any ideal containing $I$ (denoting by $m_g$ the component in $R_g$ of an element $m \in \mf$, we have $m_g \in I \subset \mf$ if $g\neq 0$, and therefore also $m_0 = m - \sum_{g \neq 0} m_g \in \mf$). It follows that the ideal $\mf^2$ is graded, and is thus a graded submodule of $\mf$. Let $\kappa=R/\mf=R_0/\mf_0$. The $\kappa$-vector space $V=\mf/\mf^2$ splits as $\bigoplus_{g \in \Char{G}} V_g$, where $V_g = \mf_g/(\mf^2)_g$. Lifting a $\kappa$-basis of each $V_g$ to $\mf_g$, we obtain a regular system of parameters $x_1,\cdots,x_n$ of $R$ and elements $g_1, \cdots , g_n \in \Char{G}$ such that $x_i \in R_{g_i}$ for each $i$. Let $J$ be the ideal of $R$ generated by those $x_i$ such that $g_i \neq 0$. To conclude the proof, it will suffice to prove that $I=J$. Clearly $J \subset I$. Since $\mf = x_1R+\cdots+x_nR$, it follows that $R_g=\mf_g = x_1 R_{g-g_1} + \cdots +x_n R_{g-g_n}$ for each $g \neq 0$. But for such $g$, we have $x_i  R_{g-g_i} \subset J$ if $g_i \neq 0$, and $x_i  R_{g-g_i} \subset \mf I$ if $g_i = 0$. Thus $I \subset \mf I + J$, which by Nakayama's lemma implies that $I=J$.
\end{proof}

Let $G$ be a finite diagonalisable group and $Y\to X$ a $G$-equivariant morphism. The deformation variety (\S\ref{sect:deformation}) inherits a $G$-action by letting the section $t$ be of weight $0 \in \Char{G}$, and the diagram of \S\ref{sect:deformation} is $G$-equivariant (for the trivial $G$-action on $\Au$).

\begin{lemma}
\label{lemm:action_char}
Let $G$ be a finite diagonalisable group acting on a variety $X$. Let $D$ be the deformation variety of the closed immersion $X^G \to X$. Then the closed immersions $X^G \to N^G$ and $X^G \times \Au \to D^G$ are isomorphisms.
\end{lemma}
\begin{proof}
It will suffice to prove that $X^G \times \Au \to D^G$ is an isomorphism. We may assume that $X$ is the spectrum of a $\Char{G}$-graded $k$-algebra $A$. The closed subscheme $X^G$ of $X$ is defined by the ideal $I\subset A$ generated by $A_g$ for $g \in \Char{G}-\{0\}$. The variety $D$ is the spectrum of the $\Char{G}$-graded $k$-algebra $R=\bigoplus_{n \in \Zz} I^{-n} t^n$, where $t \in R_0$. The closed subscheme $X^G \times \Au$ of $D$ is defined by the ideal $K \subset R$ generated by $It^{-1}$. Let $J \subset R$ be the ideal generated by $R_g$ for $g\in \Char{G} - \{0\}$. It will suffice to prove that $J=K$.

Let $g \neq 0$. For $n \geq 0$ we have $(I^{-n} t^n)_g =A_g t^n \subset It^n = (It^{-1}) t^{n+1} \subset K$, while for $n <0$ we have $(I^{-n} t^n)_g \subset I^{-n} t^n = (It^{-1})^{-n} \subset K$. Thus $R_g \subset K$, proving that $J \subset K$. 

Conversely we have $(It^{-1})_g \subset R_g \subset J$ when $g \neq 0$. The group $I_0$ is generated by the products $A_g \cdot A_{-g}$ for $g \neq 0$, hence $I_0 \subset I^2$. Thus $(It^{-1})_0=I_0t^{-1} \subset I(I t^{-1}) \subset IR \subset J$. We have proved that $It^{-1} \subset J$, hence $K \subset J$.
\end{proof}

\subsection{Free actions}
\begin{definition}
\label{def:free}
Let $G$ be a finite diagonalisable group acting on a variety $X$. We say that \emph{$G$ acts freely on $X$} if for every $g,h\in \Char{G}$ the morphism 
\[
(\varphi_*\Oc_X)_g \otimes_{\Oc_{X/G}} (\varphi_*\Oc_X)_h \to (\varphi_*\Oc_X)_{g+h}
\]
is an isomorphism ($\varphi\colon X \to X/G$ denotes the quotient morphism, see \rref{prop:quotient}).
\end{definition}

Note that if the finite diagonalisable group $G$ acts freely on $X$, then the $\Oc_{X/G}$-modules $(\varphi_*\Oc_X)_g$ are invertible, and thus the morphism $\varphi \colon X \to X/G$ is flat and finite of degree $|G|$. In addition, it follows from \dref{lemm:torsor_cartesian}{lemm:torsor_cartesian:diag} below (applied to the action morphism $G \times X \to X$) that $X \to X/G$ is a $G$-torsor in the fppf topology.

\begin{lemma}
\label{lemm:torsor_cartesian}
Let $G$ be a finite diagonalisable group and $f\colon Y \to X$ a $G$-equivariant morphism. Assume that $G$ acts freely on $X$. Then:
\begin{enumerate}[label=(\roman*),ref=\roman*]
\item \label{lemm:torsor_cartesian:free} The group $G$ acts freely on $Y$.
\item \label{lemm:torsor_cartesian:comp} For every $g \in \Char{G}$ the morphism $(f/G)^*({\varphi_X}_*\Oc_X)_g \to ({\varphi_Y}_*\Oc_Y)_g$ is an isomorphism.
\item \label{lemm:torsor_cartesian:diag} The following square is cartesian:
\[ \xymatrix{
Y\ar[r] \ar[d] & X \ar[d] \\ 
Y/G \ar[r] & X/G
}\]
\end{enumerate}
\end{lemma}
\begin{proof}
We may assume that $f$ is given by a morphism of $\Char{G}$-graded $k$-algebras $A \to B$, and moreover that for every $g \in \Char{G}$ the $A_0$-module $A_g$ is free with basis $a_g$. Then each $a_g$ is invertible in $A$, and its image $b_g \in B_g$ is invertible in $B$. The map $x \mapsto b_g \otimes (b_g^{-1}x)$ is an inverse to the morphism $B_g \otimes_{B_0} B_h \to B_{g+h}$, proving \eqref{lemm:torsor_cartesian:free}. To prove \eqref{lemm:torsor_cartesian:comp}, observe that the morphism $A_g \otimes_{A_0} B_0 \to B_g$ is bijective (an inverse is given by $x \mapsto a_g \otimes (b_g^{-1}x)$). Taking the direct sum over $g\in \Char{G}$, it follows that the morphism $A \otimes_{A_0} B_0 \to B$ is bijective, proving \eqref{lemm:torsor_cartesian:diag}.
\end{proof}

\begin{lemma}
\label{lemm:closed_free}
Let $G$ be a finite diagonalisable group acting on a variety $X$. Let $H$ be a subgroup of $G$ acting freely on $X$. Then every $G$-invariant closed subscheme of $X/H$ is of the form $Y/H$ for some $G$-invariant closed subscheme $Y$ of $X$.
\end{lemma}
\begin{proof}
Let $T$ be a $G$-invariant closed subscheme of $X/H$. Then $Y = T \times_{X/H} X$ is a $G$-invariant closed subscheme of $X$. Since $H$ acts trivially on $T$, the morphism $Y \to T$ factors through $Y/H$. Consider the following commutative diagram:
\[ \xymatrix{
Y \ar[r]_= \ar[d] &Y\ar[r] \ar[d] & X \ar[d] \\ 
Y/H \ar[r] & T \ar[r] & X/H
}\]
The square on the right is cartesian by construction, and so is the exterior square by \dref{lemm:torsor_cartesian}{lemm:torsor_cartesian:diag}. Therefore the square on the left is cartesian. Vertical morphisms are faithfully flat, and it follows by descent that $Y/H \to T$ is an isomorphism.
\end{proof}

\subsection{\texorpdfstring{$\mun$}{\textmu n}-torsors}

\begin{definition}
\label{def:Lchi}
Let $X$ be a variety with a free $\mun$-action (see \rref{def:free}). We denote by $\Lchi{X}$ the invertible $\Oc_{X/\mun}$-module $(\varphi_*\Oc_X)_1$, where $1$ denotes the canonical generator of $\Char{\mun} = \Zz/n$.  The isomorphism $\Lchi{X}^{\otimes n} \to (\varphi_*\Oc_X)_0 = \Oc_{X/G}$ allows us to define an isomorphism class $\Lchi{X}^{\otimes j}$ for every $j \in \Zz/n$. If $f \colon Y \to X$ is a $\mun$-equivariant morphism, then $(f/\mun)^*(\Lchi{X}) \simeq \Lchi{Y}$ by \dref{lemm:torsor_cartesian}{lemm:torsor_cartesian:comp}, so that we will usually write $\Lch$ instead of $\Lchi{X}$.\end{definition}

\begin{definition}
\label{def:schi}
Let $X$ be a variety with a free $\mun$-action (see \rref{def:free}). We construct below a section $\sch_X$ over $X/\mun$ of the sheaf $\Gmn$ defined in \rref{def:Gmn}. If $U$ is a $\mun$-invariant open subscheme of $X$ such that the $\Oc_{U/\mun}$-module $\Lchi{U} = \Lchi{X}|_{U/\mun}$ admits a nowhere vanishing section $a$, the image of $a^{\otimes n}$ under the morphism $(\Lchi{U})^{\otimes n} \to ((\varphi_U)_*\Oc_U)_0 = \Oc_{U/\mun}$ is an element of $H^0(U/\mun,\Gm)$, whose image in $H^0(U/\mun,\Gm/n)$ we denote by $\sch_U$. The element $\sch_U$ does not depend on the choice of $a$, hence this construction glues to define a section $\sch_X \in H^0(X/\mun,\Gmn)$.

If $f \colon Y \to X$ is a $\mun$-equivariant morphism, then $(f/\mu_n)^*(\sch_X) = \sch_Y$, so that we will usually write $\sch$ instead of $\sch_X$.
\end{definition}

\begin{remark}
Assume that $\mun$ acts freely on a variety $X$.
\begin{enumerate}[label=(\roman*)]
\item By \rref{rem:sheafification}, we have a morphism $H^1_{fppf}(X/\mun,\mun) \to H^0(X/\mun,\Gmn)$, and $\sch_X \in H^0(X/\mun,\Gmn)$ is the image of the class of the $\mun$-torsor $X \to X/\mun$.

\item The class of the invertible $\Oc_{X/\mun}$-module $\Lchi{X}$ is the image of the class of the $\mun$-torsor $X \to X/\mun$ under the morphism $H^1_{fppf}(X/\mun,\mun) \to H^1_{fppf}(X/\mun,\Gm) = \Pic(X/\mun)$ induced by the inclusion $\mun \subset \Gm$.
\end{enumerate}
\end{remark}

\begin{example}
\label{ex:Au}
Let $R=k[x]$ be the coordinate ring of $\Au$. We define a $\Zz/n$-grading on the $k$-algebra $R$ by letting $R_i$, for $i \in \Zz/n$, be the set of polynomials in $k[x]$ whose $x^j$-th coefficient vanishes unless $i = j\mod n$. This gives $\mun$-action on $\Au$. The fixed locus $(\Au)^\mun$ is the closed point $0$. The action is free on the open subscheme $X=\Au-0$, and the quotient $X/\mun$ is the spectrum of $S=k[x^n,x^{-n}]$. The element $x$ is a basis of the $S$-module $S_1$, so that the invertible $\Oc_{X/\mun}$-module $\Lchi{X}$ is trivial, and $\sch_X \in H^0(X/\mun,\Gmn)$ is the image of $x^n \in  H^0(X/\mun,\Gm)$.
\end{example}

\begin{lemma}
\label{lemm:descent_Cartier}
Let $Z \to X$ be a $\mun$-equivariant principal effective Cartier divisor, given by a section of $\Oc_X$ of weight $r \in \Zz/n=\Char{\mun}$ (see \S\ref{def:weight}). If $\mun$ acts freely on $X$, then $Z/\mun \to X/\mun$ is an effective Cartier divisor whose ideal is isomorphic to the $\Oc_{X/\mun}$-module $\Lch^{\otimes -r}$.
\end{lemma}
\begin{proof}
Let $z$ be the section. The closed immersion $Z/\mun \to X/\mun$ is given by the ideal $(z \cdot \varphi_*\Oc_X)\cap (\varphi_*\Oc_X)_0 = z \cdot (\varphi_*\Oc_X)_{-r}$ of $\Oc_{X/\mun}$, which is isomorphic to $\Lch^{\otimes -r}$.
\end{proof}

\begin{lemma}
\label{lemma:nofix_free}
Let $p$ be a prime number, and $X$ a variety with a $\mup$-action. If $X^{\mup} = \varnothing$, then $\mup$ acts freely on $X$.
\end{lemma}
\begin{proof}
We may assume that $X$ is the spectrum of a $\Zz/p$-graded $k$-algebra $A$. By assumption, the variety $X$ is covered by the $\mup$-invariant open subschemes $D(f)$ for $f \in A_m$ with $m \in \Zz/p-\{0\}$. Replacing $X$ with this cover, we may assume that there is $m\in \Zz/p-\{0\}$ such that $A_m$ contains an element $a_m$ invertible in $A$. Since $p$ is prime, it follows that for each $i \in \Zz/p$ the set $A_i$ contains an element $a_i$ invertible in $A$ (a power of $a_m$). Thus $A_i \otimes_{A_0} A_j \to A_{i+j}$ is an isomorphism for each $i,j \in \Zz/p$ (an inverse is given by $a \mapsto a_i \otimes (a_i^{-1}a)$).
\end{proof}

\subsection{Constant finite abelian groups}
\label{sect:constant}
Let $\Gamma$ be an ordinary finite abelian group. The group structure on $\Gamma$ induces a Hopf algebra structure on the $k$-algebra $k^\Gamma$ of maps of sets $\Gamma \to k$ (see  e.g. \cite[\S2.3]{Waterhouse}). The corresponding algebraic group $\Gamma_k=\Spec k^\Gamma$ is such that $\Hom_k(\Spec R,\Gamma_k)=\Gamma$ for any commutative connected $k$-algebra $R$. An action of the ordinary group $\Gamma$ on a variety $X$  (in the sense of \S\ref{sect:results}) is the same thing as an action of the algebraic group $\Gamma_k$ and the notions of equivariant morphisms coincide.

\begin{lemma}
\label{lemm:constant_is_diag}
Let $\Gamma$ be an ordinary finite abelian group. If $k$ contains a root of unity whose order is the exponent of $\Gamma$, then the algebraic group $\Gamma_k$ is finite diagonalisable.
\end{lemma}
\begin{proof}
Let $\Gamma^*$ be the ordinary abelian group of group morphisms $\Gamma \to k^\times$.  We prove that the algebraic group $\Gamma_k$ is isomorphic to $\Diag(\Gamma^*)$. The inclusion $\Gamma^* \subset k^\Gamma$ induces a $k$-linear morphism $u\colon k[\Gamma^*] \to k^\Gamma$. One checks that the subset $\Gamma^*$ of $k^\Gamma$ consists of group-like elements, and that $u$ is  a morphism of Hopf algebras. By linear independence of characters \cite[\S2.2, Lemma]{Waterhouse}, the morphism $u$ is injective. By the assumption on $k$, we have $|\Gamma^*| = |\Gamma|$, hence $\dim_k k[\Gamma^*] = \dim_k k^\Gamma $, so that $u$ is an isomorphism.
\end{proof}

\section{The Rost operation}
\label{sect:operation}

\subsection{The operation \texorpdfstring{$\rho$}{\textrho}}
From now on, the letter $p$ will denote a prime number, and $M$ a cycle module over $k$ such that $p\cdot M=0$.

Let $X$ be a variety with a $\mup$-action. We endow $\Au$ with the $\mup$-action of \rref{ex:Au}, and consider the $\mup$-invariant open subscheme of $X \times \Au$
\[
X^\circ = (X \times \Au) - (X \times \Au)^{\mup} = (X \times \Au) - (X^{\mup} \times 0).
\]
By \rref{lemma:nofix_free}, the $\mup$-action on $X^\circ$ is free. If $f\colon Y \to X$ is a $\mup$-equivariant morphism such that $f^{-1}(X^\mup)=Y^\mup$ (set-theoretically), there is an induced morphism $f^\circ \colon Y^\circ \to X^\circ$.

\begin{definition}
\label{def:rho}
Let $X$ be a variety with a $\mup$-action. By \dref{lemm:quotient-map}{lemm:quotient-map:closed}, the morphism $X^\mup = X^\mup \times 0 \to (X \times \Au)/\mup$ is a closed immersion, and its open complement is $X^\circ/\mup$ by \dref{prop:quotient}{prop:quotient:open}. We consider the connecting homomorphism
\[
\partial_X \colon A(X^\circ/\mup,M) \to A(X^\mup,M),
\]
the invertible $\Oc_{X^\circ/\mup}$-module $\Lch=\Lchi{X^\circ}$ (see \rref{def:Lchi}), the section $\sch=\sch_{X^\circ} \in H^0(X^\circ/\mup,\Gmp)$ (see \rref{def:schi}), and define a morphism
\[
\rho_X = \partial_X \circ \{\sch\} \circ c(-\Lch) \colon A(X^\circ/\mup,M) \to A(X^\mup,M).
\]
\end{definition}

The operation $\rho$ is compatible with restrictions to open subschemes:

\begin{proposition}
\label{prop:func_rho:open}
Let $u\colon U \to X$ be the immersion of a $\mup$-invariant open subscheme. Then  $(u^\mup)^* \circ \rho_X = \rho_U \circ (u^\circ/\mup)^*$.
\end{proposition}
\begin{proof}
By \dref{prop:quotient}{prop:quotient:open}, the diagram
\[ \xymatrix{
U^\mup\ar[r] \ar[d] & U\times \Au \ar[d] \ar[r] & (U\times \Au)/\mup \ar[d]\\ 
X^\mup \ar[r] & X \times \Au\ar[r] & (X\times \Au)/\mup
}\]
has cartesian squares, and vertical arrows are open immersions. We have, as morphisms $A(X^\circ/\mup,M) \to A(U^\mup,M)$,
\begin{align*}
(u^\mup)^* \circ \rho_X
&= (u^\mup)^* \circ \partial_X \circ \{\sch\} \circ c(-\Lch) \\ 
&= \partial_U \circ (u^\circ/\mup)^* \circ \{\sch\} \circ c(-\Lch) && \text{by \cite[(4.4.2)]{Rost-Chow}}\\
&= \partial_U \circ \{\sch\} \circ (u^\circ/\mup)^* \circ c(-\Lch)&&\text{by \dref{lemm:Gmn}{lemm:Gmn:flat}}\\
&= \partial_U \circ \{\sch\} \circ c(-\Lch) \circ (u^\circ/\mup)^* &&\text{by \dref{lemm:c1}{lemm:c1:flat}}\\
&= \rho_U \circ (u^\circ/\mup)^*.&&\qedhere
\end{align*}
\end{proof}

\begin{corollary}
\label{prop:rho_open}
Let $u\colon U \to X$ and $v\colon V\to X$ be two $\mup$-invariant open immersions such that $U \cup V = X$ and $U^\mup \cap V^\mup = \varnothing$. Then
\[
\rho_X = (u^\mup)_* \circ \rho_U \circ (u^\circ/\mup)^* + (v^\mup)_* \circ \rho_V \circ (v^\circ/\mup)^*.
\]
\end{corollary}
\begin{proof}
Indeed $X^\mup$ is the disjoint union of the open subschemes $U^\mup$ and $V^\mup$, so that $(u^\mup)_* \circ (u^\mup)^* + (v^\mup)_* \circ (v^\mup)^*$ is the identity of $A(X^\mup,M)$, and by \rref{prop:func_rho:open}
\begin{align*}
\rho_X 
&= \big( (u^\mup)_* \circ (u^\mup)^* + (v^\mup)_* \circ (v^\mup)^*\big) \circ \rho_X  \\ 
&= (u^\mup)_* \circ \rho_U \circ (u^\circ/\mup)^* + (v^\mup)_* \circ \rho_V \circ (v^\circ/\mup)^*.\qedhere
\end{align*}
\end{proof}

The next proposition expresses the compatibility of $\rho$ with proper pushforwards.

\begin{proposition}
\label{prop:func_rho:proper} Let $f\colon Y \to X$ be a proper $\mup$-equivariant morphism. Then the following diagram commutes.
\[ \xymatrix{
A(Y^\circ/\mup,M)\ar[rr]^{(u/\mup)^*} \ar[d]_{\rho_Y} && A(Y'/
\mup,M) \ar[rr]^{(h/\mup)_*} && A(X^\circ/\mup,M) \ar[d]^{\rho_X} \\ 
A(Y^\mup,M) \ar[rrrr]^{(f^\mup)_*} &&&& A(X^\mup,M)
}\]
Here $Y'=(Y\times \Au) - (f^{-1}(X^\mup) \times 0)$, and $u\colon Y' \to Y^\circ, h \colon Y' \to X^\circ$ are the induced morphisms.

In particular if $f^{-1}(X^\mup)=Y^\mup$ (set-theoretically), then $(f^\mup)_* \circ \rho_Y =  \rho_X \circ (f^\circ/\mup)_*$.
\end{proposition}
\begin{proof}
Let $Z$ be the inverse image in $(Y \times \Au)/\mup$ of the closed subscheme $X^\mup=X^\mup \times 0$ of $(X \times \Au)/\mup$. Its open complement is $Y'/\mup$ by \dref{prop:quotient}{prop:quotient:open}. Consider the diagram
\[ \xymatrix{
A(Y^\circ/\mup,M)\ar[r]^{(u/\mup)^*} \ar[d]_{\partial_Y} & A(Y'/\mup,M) \ar[d] \ar[r]^{(h/\mup)_*} & A(X^\circ/\mup,M) \ar[d]^{\partial_X} \\ 
A(Y^\mup,M) \ar[r] & A(Z,M) \ar[r] & A(X^\mup,M)
}\]
where vertical arrows are connecting homomorphisms, and lower horizontal ones are proper pushforwards (the lower composite is $(f^\mup)_*$). The square on the right commutes by \cite[(4.4.1)]{Rost-Chow} and so does the one on the left by \rref{lemm:connecting}. Thus, as morphisms $A(Y^\circ/\mup,M) \to A(X^\mup,M)$,
\begin{align*}
(f^\mup)_* \circ \rho_Y   
&= (f^\mup)_* \circ \partial_Y \circ \{\sch\} \circ c(-\Lch)\\ 
&= \partial_X \circ (h/\mup)_* \circ (u/\mup)^* \circ \{\sch\} \circ c(-\Lch) \\
&= \partial_X \circ \{\sch\} \circ (h/\mup)_* \circ (u/\mup)^* \circ c(-\Lch)&&\text{by \dref{lemm:Gmn}{lemm:Gmn:flat}, \dref{lemm:Gmn}{lemm:Gmn:pushforward}}\\
&= \partial_X \circ \{\sch\} \circ c(-\Lch) \circ (h/\mup)_* \circ (u/\mup)^* && \text{by \dref{lemm:c1}{lemm:c1:flat}, \dref{lemm:c1}{lemm:c1:proper}}\\
&= \rho_X \circ (h/\mup)_* \circ (u/\mup)^*.&&\qedhere
\end{align*}
\end{proof}

The next lemma describes the operation $\rho$ in the case of a trivial $\mup$-action.
\begin{lemma}
\label{lemm:trivial_action}
Let $X$ be a variety with trivial $\mup$-action. Then the composite
\[
A(X,M) \xrightarrow{a^*} A(X^\circ/\mup,M) \xrightarrow{\rho_X} A(X,M)
\]
is the identity, where $a \colon X^\circ/\mup = X \times ((\Au-0)/\mup)\to X$ is the first projection.
\end{lemma}
\begin{proof}
Let $k[x]$ be the coordinate ring of $\Au$. In view of \rref{ex:Au}, we have
\[
\rho_X = \partial_X \circ \{\sch\} \circ c(-\Lch) = \partial_X \circ \{x^p\}.
\]
Now $0 \to \Au/\mup$ is the principal effective Cartier divisor given by the global section $x^p$, so that by \cite[(4.5)]{Rost-Chow} the composite $\partial_X \circ \{x^p\} \circ a^*$ is the identity of $A(X,M)$.
\end{proof}

\subsection{The degree formula}

\begin{definition}
\label{def:varrho}
Let $X$ be a variety with a $\mup$-action. Letting $M=K_*/p$ be the modulo $p$ Milnor $K$-theory cycle module in \rref{def:rho}, we define
\[
\varrho(X) = \rho_X[X^\circ/\mup] \in A(X^\mup,K_0/p)=\CH(X^\mup)/p,
\]
and denote by $\varrho_n(X) \in \CH_n(X^\mup)/p$ its component of degree $n$, for every $n \in \mathbb{N}$.

\end{definition}

\begin{proposition}[Rost's degree formula]
\label{prop:df}
Let $f\colon Y \to X$ be a proper $\mup$-equivariant morphism with a degree (see \S\ref{def:deg}). Then
\[
(f^\mup)_*\varrho(Y) = \deg f \cdot \varrho(X) \in \CH(X^\mup)/p.
\]
\end{proposition}
\begin{proof}
The diagram below has cartesian squares by \dref{lemm:torsor_cartesian}{lemm:torsor_cartesian:diag}
\[ \xymatrix{
Y'/\mup \ar[d]_{h/\mup} &&Y'\ar[r] \ar[d]_h \ar[ll]& Y \ar[d]^f \\ 
X^\circ/\mup && X^\circ \ar[r] \ar[ll]_{\varphi_{X^\circ}}& X
}\]
The morphism $h$ has degree $\deg f$ by \cite[Proposition 1.7, Lemma 1.7.1]{Ful-In-98}. The flat pullback $(\varphi_{X^\circ})^* \colon \Zo(X^\circ/\mup) \to \Zo(X^\circ)$ is injective, because $(\varphi_{X^\circ})_* \circ (\varphi_{X^\circ})^* = p\cdot \id$, and $\Zo(X^\circ/\mup)$ has no $p$-torsion. Applying once again \cite[Proposition 1.7, Lemma 1.7.1]{Ful-In-98}, we deduce that $h/\mup$ has degree $\deg f$. The statement then follows from \rref{prop:func_rho:proper}.
\end{proof}

\begin{corollary}
\label{cor:df}
Let $X$ be a projective variety without zero-dimensional connected component. If $\mup$ acts on $X$, then $\deg \varrho(X) =0 \in \Fp$.
\end{corollary}
\begin{proof}
Apply \rref{prop:df} to the proper $\mup$-equivariant morphism $X \to \Spec k$, which has degree zero. 
\end{proof}

\subsection{Computation of the first term}

\begin{lemma}
\label{lemm:rho_rep}
Let $\Vc$ be a finite-dimensional $\mup$-representation over $k$ such that $\Vc_0=0$, and $Y$ an irreducible variety with trivial $\mup$-action. Consider the variety $V= \Spec(\Sym_k (\Vc^\vee))$ with its induced $\mup$-action (here $\Sym_k$ denotes the symmetric algebra). Then $\varrho(Y \times V) \in \CH(Y)/p$ is a non-zero multiple of $[Y]$.
\end{lemma}
\begin{proof}
We proceed by induction on the dimension $n$ of $\Vc$. The statement follows from \rref{lemm:trivial_action} if $n=0$. If $n >0$, we may decompose the $\mup$-representation $\Vc$ as $\Wc \oplus \Lc$, with $\Lc$ one-dimensional. Since $\Vc_0 = 0$, we must have $\Lc=\Lc_r$ for some $r\in \Zz/p-\{0\}$. Let $W=\Spec(\Sym_k (\Wc^\vee))$. The closed immersion $i\colon Y \times W \to Y \times V$ is a principal effective Cartier divisor given by a section of weight $-r$ (see \S\ref{def:weight}), hence the same is true for the closed immersion $i^\circ$. By \rref{lemm:descent_Cartier}, the closed immersion $i^\circ/\mup$ is an effective Cartier divisor whose ideal is isomorphic to $\Lch^{\otimes r}$. Thus we have in $\CH(Y)/p$,
\begin{align*}
\varrho(Y \times W)&= (i^\mup)_* \circ \rho_{Y \times W}[(Y \times W)^\circ/\mup]&& \text{ since $i^\mup=\id_Y$}\\
&= \rho_{Y\times V} \circ (i^\circ/\mup)_*[(Y \times W)^\circ/\mup]&& \text{ by \rref{prop:func_rho:proper}}\\
&= \rho_{Y \times V} \circ c_1(\Lch^{\otimes -r})[(Y \times V)^\circ/\mup] && \text{ by \cite[Theorem 3.2(f)]{Ful-In-98}}\\
&= -r \cdot \rho_{Y \times V} \circ c_1(\Lch)[(Y \times V)^\circ/\mup]&& \text{ by \cite[Proposition 2.5(e)]{Ful-In-98}}\\
&= r \cdot (\varrho(Y \times V) - \varrho_{\dim Y +n}(Y \times V)) && \text{ since $c(-\Lch) = \id - c(-\Lch) \circ c_1(\Lch)$}\\
&= r \cdot \varrho(Y \times V),&&
\end{align*}
since $\varrho_{\dim Y + n}(Y \times V) \in \CH_{\dim Y+ n}(Y)/p=0$. We conclude by induction.
\end{proof}

\begin{lemma}
\label{lemm:normal}
Let $X$ be an equidimensional variety with a $\mup$-action. Let $N$ be the normal cone of the closed immersion $X^\mup \to X$. Then $X^\mup = N^\mup$ (see \rref{lemm:action_char}), and
\[
\varrho(X) = \varrho(N) \in \CH(X^\mup)/p.
\]
\end{lemma}
\begin{proof}
The affine line will be denoted by $\Au$ when it is endowed with the $\mup$-action of \rref{ex:Au}, and by $\Auc$ when $\mup$ acts trivially. Let $D$ be the deformation variety of the closed immersion $X^\mup \to X$ and consider the commutative diagram (see \S\ref{sect:deformation}) 
\[ \xymatrix{
X^\mup \ar[r] \ar[d] & N^\mup \ar[d] \ar[r] & N \times \Au \ar[d] \ar[r] &(N \times \Au)/\mup \ar[d] \ar[r] & 0 \ar[d]\\ 
X^\mup \times \Au\ar[r] & D^\mup \ar[r] & D \times \Au \ar[r] &(D \times \Au)/\mup \ar[r] & \Auc
}\]
The observation at the end of \S\ref{sect:quotients} shows that the square on the right is cartesian, and it follows easily that the remaining squares are also cartesian. The left horizontal arrows are isomorphisms by \rref{lemm:action_char}. Thus by \cite[(11.6)]{Rost-Chow} the square
\begin{equation}
\label{eq:anti}
\begin{aligned}
\xymatrix{
A((X^\circ/\mup) \times (\Auc -0),K_*/p) \ar[rr]^{\delta} \ar[d]_{\partial_{X \times (\Auc-0)}}  && A(N^\circ/\mup,K_*/p) \ar[d]^{\partial_N} \\ 
A(X^\mup \times (\Auc -0),K_*/p) \ar[rr]^{\Delta} &&  A(X^\mup,K_*/p)
 }
\end{aligned}
\end{equation}
anticommutes, where $\delta$, resp.\ $\Delta$, is the connecting homomorphism associated with the closed immersion $N^\circ/\mup \to D^\circ/\mup$, resp.\ $X^\mup=X^\mup \times 0 \to X^\mup \times \Auc$.

Denote by $q \colon (X \times \Au)/\mup\times (\Auc -0) \to (X \times \Au)/\mup$ the first projection and by $t\in H^0(\Auc,\Gm)$ the section induced by the identification $\Auc = \Spec k[t]$. We have, as morphisms $A(X^\circ/\mup,K_*/p) \to A(X^\mup,K_*/p)$,
\begin{align*}
\rho_X 
&= \partial_X \circ \{\sch\} \circ c(-\Lch)\\ 
&= \Delta \circ \{t\} \circ (q|_{X^\mup})^* \circ \partial_X \circ \{\sch\} \circ c(-\Lch)&&\text{by \cite[(4.5)]{Rost-Chow}}\\
&= \Delta \circ \{t\} \circ \partial_{X \times (\Auc -0)} \circ (q|_{X^\circ/\mup})^* \circ \{\sch\} \circ c(-\Lch)&&\text{by \cite[(4.4.2)]{Rost-Chow}}\\
&= -\Delta \circ \partial_{X \times (\Auc -0)} \circ \{t\} \circ (q|_{X^\circ/\mup})^* \circ \{\sch\} \circ c(-\Lch)&&\text{by \cite[(4.3.2)]{Rost-Chow}}\\
&= \partial_N \circ \delta \circ \{t\} \circ (q|_{X^\circ/\mup})^* \circ \{\sch\} \circ c(-\Lch)&&\text{as \eqref{eq:anti} anticommutes}\\
&= \partial_N \circ \{\sch\} \circ \delta \circ \{t\}\circ (q|_{X^\circ/\mup})^* \circ c(-\Lch) &&\text{by \rref{lemm:Gmn}}\\
&= \partial_N \circ \{\sch\} \circ c(-\Lch) \circ \delta \circ \{t\}\circ (q|_{X^\circ/\mup})^*&&\text{by \rref{lemm:c1}}\\
&=\rho_N \circ \delta \circ \{t\}\circ (q|_{X^\circ/\mup})^*&&
\end{align*}

The morphism $D^\circ \to \Auc$ is flat and factors through the faithfully flat morphism $D^\circ \to D^\circ /\mup$. Thus the morphism $f \colon D^\circ/\mup \to \Auc$ is flat \cite[(2.2.11.iv)]{ega-4-2}. Consider the connecting homomorphism
\[
\varsigma \colon A(\Auc -0,K_*/p) \to A(0,K_*/p)
\]
associated to the closed point $0 \to \Auc$. Using repeatedly \cite[Lemma 1.7.1]{Ful-In-98}, we have in $\CH(N^\circ/\mup)/p$
\begin{align*}
\delta \circ \{t\} \circ (q|_{X^\circ/\mup})^* [X^\circ/\mup]
&= \delta \circ \{t\} [X^\circ/\mup \times (\Auc-0)]\\ 
&= \delta \circ \{t\} \circ  (f|_{\Auc -0})^*[\Auc-0]\\
&= \delta \circ (f|_{\Auc -0})^*\circ \{t\}[\Auc-0] && \text{ by \cite[(4.3.1)]{Rost-Chow}}\\
&= (f|_0)^* \circ \varsigma \circ \{t\}[\Auc-0] && \text{ by \cite[(4.4.2)]{Rost-Chow}}\\
&= (f|_0)^*[0]=[N^\circ /\mup]&& \text{ by \cite[(4.5)]{Rost-Chow}.}
\end{align*}
Combining these two computations, we obtain in $\CH(X^\mup)/p$,
\[
\varrho(X) = \rho_X [X^\circ /\mup] = \rho_N \circ \delta \circ \{t\}\circ (q|_{X^\circ/\mup})^* [X^\circ /\mup] = \rho_N[N^\circ/\mup]=\varrho(N).\qedhere
\]
\end{proof}

\begin{proposition}
\label{prop:nonvanishing}
Let $X$ be a variety with a $\mup$-action. Assume that $X^\mup$ is irreducible of dimension $n$ and that $X$ is regular at the generic point of $X^\mup$. Then
\[
\varrho_n(X) \neq 0 \in \CH_n(X^\mup)/p=\Fp.
\]
\end{proposition}
\begin{proof}
Observe that while proving the statement, we may replace $X$ by any $\mup$-invariant open subscheme $U$ of $X$ meeting $X^\mup$. Indeed the restriction morphism $\CH_n(X^\mup)/p \to \CH_n(U^\mup)/p$ sends $\varrho_n(X)$ to $\varrho_n(U)$ by \rref{prop:func_rho:open}.

Let $R$ be a regular open subscheme of $X$ meeting $X^\mup$, and $R'$ the  maximal $\mup$-invariant open subscheme of $X$ contained in $R$. Then $R'$ meets $X^\mup$ by \tref{prop:open_action}{prop:open_action:sat}{prop:open_action:sat:inv}. Replacing $X$ with $R'$, we may assume that $X$ is regular. Then $X^\mup$ is regular by \rref{lemm:fixed_reg}. Let $q \colon N \to X^\mup$ be the normal bundle of the regular closed immersion $X^\mup \to X$, and $\Nc$ its $\Oc_{X^\mup}$-module of sections. The $\mup$-action on $X$ induces a decomposition $\Nc = \bigoplus_{i \in \Zz/p} \Nc_i$ as $\Oc_{X^\mup}$-modules. Further shrinking $X$, we may assume that each $\Oc_{X^\mup}$-module $\Nc_i$ is free. Then there is a $\mup$-representation $\Vc$ over $k$ and a $\mup$-equivariant isomorphism $\Nc = \Vc \otimes_k \Oc_{X^\mup}$. Since $N^\mup = X^\mup$ by \rref{lemm:action_char}, we have $\Nc_0=\Vc_0=0$. By \rref{lemm:rho_rep}, there is an element $u \in (\Fp)^\times$ such that $\varrho(N) = u \cdot [X^\mup]$ in $\CH(X^\mup)/p$. We conclude using \rref{lemm:normal}.
\end{proof}

\begin{remark}
\label{rem:weight}
The element $u \in (\Fp)^\times$ such that $\varrho_n(X) = u\cdot[X^\mup]$ in \rref{prop:nonvanishing} can be explicitly computed in terms of the $\mup$-action on the normal bundle to $X^\mup$ in $X$. Indeed, its fiber at the generic point of $X^\mup$ is a $\mup$-representation $\Vc$ over the field $K=k(X^\mup)$, and it follows from the proof of \rref{lemm:rho_rep} that
\[
u = \prod_{i \in \Zz/p-\{0\}} i^{-\dim_K \Vc_i} \in (\Fp)^\times.
\]
\end{remark}

\section{Equivariance of the operation}
\label{sect:equiv}

\subsection{Equivariant Cycles}
When $G$ is a finite constant group acting on a variety $X$, one may define the set of $G$-equivariant cycles as the equaliser of the two pullback maps $\Zo(X) \to \Zo(G \times X)$. This definition is however inappropriate when $G$ is an arbitrary algebraic group (consider the action of $\mup$ on itself in characteristic $p$), and we will instead use the following

\begin{definition}
Let $G$ be an algebraic group acting on a variety $X$. We denote by $\Zo_G(X) \subset \Zo(X)$ the subgroup generated by classes of equidimensional $G$-invariant closed subschemes of $X$.
\end{definition}

\begin{lemma}
\label{lemm:equ_cover}
Let $G$ be an algebraic group acting on a variety $X$, and $U_i$ a family of $G$-invariant open subschemes covering $X$. Then a cycle in $\Zo(X)$ belongs to $\Zo_G(X)$ if and only if its restriction to $\Zo(U_i)$ belongs to $\Zo_G(U_i)$ for every $i$.
\end{lemma}
\begin{proof}
When $v\colon V \to Y$ is a $G$-invariant open immersion, we have $v^*\Zo_G(Y) \subset \Zo_G(V)$. Since the scheme-theoretic closure of a $G$-invariant closed subscheme of $V$ is a $G$-invariant closed subscheme of $Y$, we also have $v_*\Zo_G(V) \subset \Zo_G(Y)$. 

Thus one implication is clear. Let now $z\in \Zo(X)$ be such that $z|_{U_i} \in \Zo_G(U_i)$ for every $i$. Let $V$ be a $G$-invariant open subscheme of $X$ such that $z|_V \in \Zo_G(V)$ and $V \neq X$. We construct a $G$-invariant open subscheme $W$ of $X$ such that $V \subsetneq W$ and $z|_W \in \Zo_G(W)$. By quasi-compactness of $X$, this will prove that $z \in \Zo_G(X)$, concluding the proof of the lemma. We may find $i$ such that $U_i \not \subset V$. Let $W=V \cup U_i$. Then
\[
z|_W = v_*(z|_V) + u_*(z|_{U_i}) - c_*(z |_{V \cap U_i}),
\]
where $v\colon V \to W$ and $u \colon U_i \to W$ and $c \colon V \cap U_i \to W$ are the open immersions. Using the above observations, we deduce that $z|_W \in \Zo_G(W)$, as required.
\end{proof}

\begin{definition}
A non-empty variety with a $G$-action in which every non-empty $G$-invariant open subscheme is dense will be called \emph{$G$-irreducible}.
\end{definition}

\begin{lemma}
\label{lemm:Girred}
Let $G$ be an algebraic group acting on a variety $X$. Then $\Zo_G(X)$ is generated by classes of equidimensional $G$-irreducible closed subschemes of $X$.
\end{lemma}
\begin{proof}
For every equidimensional $G$-invariant closed subscheme $Y$ of $X$, we prove that the class $[Y] \in \Zo(Y)$ may be written as the sum of classes of $G$-irreducible closed subschemes of $X$. We may assume that $Y$ is non-empty and not $G$-irreducible. Let $U$ be a non-empty and non-dense $G$-invariant open subscheme of $Y$. Let $Y_0$ be its scheme-theoretic closure in $Y$, and $Y_1$ the scheme-theoretic closure of $Y-Y_0$ in $Y$. The set of generic points of $Y$ is the disjoint union of the sets of generic points of $Y_0$ and $Y_1$, and the multiplicities coincide. It follows that $[Y] = [Y_0] +[Y_1]$ in $\Zo(Y)$. Since $Y_0 \neq Y$ (as $U$ is not dense in $Y$) and $Y_1 \neq Y$ (as $U \cap Y_1 = \varnothing$ and $U \neq \varnothing$), and $Y_0,Y_1$ are equidimensional and $G$-invariant, we may conclude using noetherian induction.
\end{proof}

We now state a technical fact used in the proof of \rref{prop:deg_Ginv}.
\begin{lemma}
\label{lemm:proj_Ginv}
Let $G$ be an algebraic group, and $i \colon Y \to X$ the immersion of a $G$-invariant closed subscheme. For any $\alpha \in \Zo_G(X)$, the cycle $i_*(\alpha|_Y) \in \Zo(X)$ is a $\Zz$-linear combination of classes of equidimensional $G$-invariant closed subschemes of $X$ supported on $Y$.
\end{lemma}
\begin{proof}
By \rref{lemm:Girred} we may assume that $\alpha =[X]$ and that $X$ is equidimensional and $G$-irreducible. Then either $i$ is surjective and $i_*([X]|_Y)=[X]$, or $Y$ contains no generic point of $X$ and $[X]|_Y=0$. In either case, the statement is true.
\end{proof}

\subsection{Representing \texorpdfstring{$\varrho(X)$}{\textrho(X)} by equivariant cycles}
The next statement asserts that the divisor class of a rational function on which a finite diagonalisable group acts through a character is represented by an equivariant cycle.

\begin{lemma}
\label{lemm:div_weight}
Let $G$ be a finite diagonalisable group acting on a variety $X$, and $u\colon U \to X$ the immersion of an equidimensional $G$-invariant open subscheme. Let $a \in  H^0(U,\Gm)$ be homogeneous (as a section of $\Oc_U$, see \S\ref{def:weight}). Then $[U]$ is mapped to an element of $\Zo_G(X)$ under the composite
\[
\Zo(U) \xrightarrow{\{a\}} C(U,K_1) \xrightarrow{u_*} C(X,K_1) \xrightarrow{d} \Zo(X).
\]
\end{lemma}
\begin{proof}
Replacing $X$ with the scheme-theoretic closure of $U$, we may assume that $U$ contains all associated points of $X$ and that $X$ is equidimensional. In view of \rref{lemm:equ_cover}, we may replace $X$ with a cover by $G$-invariant open subschemes, and thus assume that $X$ is the spectrum of a $\Char{G}$-graded $k$-algebra $A$. We are going to find an element $f \in A_0$, which is a nonzerodivisor in $A$ and is such that $D(f) \subset U$. First observe that $U$ admits a $G$-invariant closed complement in $X$ (let $R$ be the reduced closed complement of $U/G$ in $X/G$, and take $R \times_{X/G} X$); it is defined by some graded ideal $J$ of $A$. We claim that $J_0$ is contained in no associated prime of $A$. Indeed let $\pf$ be a prime of $A$ such that $J_0 \subset \pf$. If $x \in J_g$ for some $g\in \Char{G}$, then $x^{|G|} \in J_0\subset \pf$, hence $x \in \pf$. It follows that $J \subset \pf$. Thus $\pf$ cannot be an associated prime of $A$ (as they all lie in $U$), which proves the claim. By prime avoidance, we may find $f \in J_0$ such that $f$ is in none of the finitely many primes $\pf \cap A_0$ of $A_0$ for $\pf$ an associated prime of $A$. Then $f$ has the required properties.

Since the $G$-invariant open subscheme $D(f)$ of $X$ is dense (being the complement of an effective Cartier divisor) and contained in $U$, the open immersion $v \colon D(f) \to U$ is dense, hence by \cite[(4.2.1) and (4.1.1)]{Rost-Chow}
\[
u_* \circ \{a\} [U] = u_* \circ \{a\} \circ v_*[D(f)] = u_* \circ v_* \circ \{a\}[D(f)] = (u \circ v)_* \circ \{a\}[D(f)],
\]
so that we may replace $U$ with $D(f)$ while proving the lemma. Then $a=b/f^n\in A[1/f]^\times$ for some integer $n\geq 0$ and $b \in A$. We view $A$ as a subalgebra of $A[1/f]$ (as $f$ is a nonzerodivisor in $A$). Then $b$ is a nonzerodivisor in $A$, as is any element of $A[1/f]^\times \cap A$. Since $f \in A_0$, the element $b=af^n \in A$ is homogeneous (of the same weight as $a$). By \cite[Lemma~1.7.2]{Ful-In-98} we have in $\Zo(X)$
\[
d\circ u_*\circ \{a\}[U] = d\circ u_*\circ \{b\}[U] - n\cdot d\circ u_*\circ \{f\}[U]= [Z(b)] - n \cdot [Z(f)].
\]
The closed subschemes $Z(b)$ and $Z(f)$ of $X$ (defined in \S\ref{sect:D}) are $G$-invariant by \S\ref{def:weight} and equidimensional (being effective Cartier divisors in the equidimensional variety $X$), hence the above cycle belongs to $\Zo_G(X)$.
\end{proof}

\begin{lemma}
\label{lemm:ZG}
Let $G$ be a finite diagonalisable group containing $\mun$, and $X$ a variety with a $G$-action. Assume that $\mun$ acts freely on $X$ (see \rref{def:free}). Then:
\begin{enumerate}[label=(\roman*),ref=\roman*]
\item \label{lemm:ZG:action} The $G$-action on $X$ induces a $G$-action on $X/\mun$ and a $G$-equivariant structure on the $\Oc_{X/\mun}$-module $\Lch$ defined in \rref{def:Lchi}.

\item \label{lemm:ZG:L} The $\Oc_{X/\mun}$-module $\Lch$ is locally generated by a $\Char{G}$-homogeneous section.

\item \label{lemm:ZG:c1} The image of $\Zo_G(X/\mun)$ in $\CH(X/\mun)$ is stable under $c_1(\Lch)$.

\item \label{lemm:ZG:s} Let $u\colon X \to \overline{X}$ be a $G$-invariant open immersion, and consider the section $\sch \in H^0(X/\mun,\Gmn)$ defined in \rref{def:schi}. Then the composite
\[
\Zo(X/\mun)/n \xrightarrow{\{\sch\}} C(X/\mun,K_1/n) \xrightarrow{u_*} C(\overline{X}/\mun,K_1/n) \xrightarrow{d} \Zo(\overline{X}/\mun)/n
\]
maps the image of $\Zo_G(X/\mun)$ into the image of $\Zo_G(\overline{X}/\mun)$.
\end{enumerate}
\end{lemma}
\begin{proof}
\eqref{lemm:ZG:action}: The morphism $X \xrightarrow{\varphi} X/G$ is $\mun$-equivariant, hence factors as $X \xrightarrow{\psi} X/\mun \xrightarrow{\lambda} X/G$.
 For $i \in \Zz/n$, we have $\lambda_* ((\psi_* \Oc_X)_i) = \bigoplus_{\alpha(g)=i} (\varphi_*\Oc_X)_g$, where $\alpha \colon \Char{G} \to \Char{\mun}=\Zz/n$ is induced by the inclusion $\mun \subset G$. The first part of \eqref{lemm:ZG:action} follows from the case $i=0$, \rref{lemm:action_grading} and \rref{prop:quotient_graded}. The second part follows from the case $i=1$.

\eqref{lemm:ZG:L}: We may assume that $X$ is the spectrum of a $\Char{G}$-graded $k$-algebra $A$, and that the $A_0$-module $A_1$ is free with basis $a$ (here $0,1 \in \Zz/n=\Char{\mun}$). For $g \in \Char{G}$, denote by $a_g$ the component of $a$ in $A_g$. Then the $G$-invariant (by \S\ref{def:weight}) open subschemes $D_{\Lch}(a_g)$ for $g \in \Char{G}$ cover $X/\mun$ by \eqref{eq:D}, and $\Lch|_{D_{\Lch}(a_g)}$ is generated by the section $a_g$.

\eqref{lemm:ZG:c1}: In view of \rref{lemm:Girred} and \rref{lemm:closed_free}, it will suffice to prove that $c_1(\Lch)[X/\mun] \in \CH(X/\mun)$ is represented by a cycle in $\Zo_G(X/\mun)$ under the assumption that $X/\mun$ is equidimensional and $G$-irreducible. By \eqref{lemm:ZG:L}, we may find a cover of $X/\mun$ by $G$-invariant open subschemes $V_i$, and for each $i$ a nowhere vanishing section $l_i$ of $\Lch|_{V_i}$ of weight $g_i \in \Char{G}$. Let us fix an index $j$ such that $V_j \neq \varnothing$, and thus $V_j$ is dense in $X/\mun$. For each $i$, let $s_i \in H^0(V_i \cap V_j,\Gm)$ be the element such that $l_j = s_i l_i$ as a section of $\Lch$ on $V_i \cap V_j$. The section $s_i$ of $\Oc_{V_i \cap V_j}$ has weight $g_j-g_i \in \Char{G}$. By \rref{lemm:c1_merom} and \rref{lemm:div_weight}, the cycle class $c_1(\Lch)[X/\mun] \in \CH(X/\mun)$ is represented by a cycle in $\Zo(X/\mun)$ whose restriction to each $V_i$ belongs to $\Zo_G(V_i)$. We conclude using \rref{lemm:equ_cover}.

\eqref{lemm:ZG:s}: Any $G$-irreducible closed subscheme of $X/\mun$ is of the form $Y/\mun$ for some $G$-invariant closed subscheme $Y$ of $X$ by \rref{lemm:closed_free}. Let $\overline{Y}$ be the scheme-theoretic closure of $Y$ in $\overline{X}$. Replacing $X\to \overline{X}$ with $Y \to \overline{Y}$, it will suffice in view  of \rref{lemm:Girred} to prove that $d \circ u_* \circ \{\sch\}[X/\mun]$ is the image of a cycle in $\Zo_G(\overline{X}/\mun)$ under the assumption that $X/\mun$ is equidimensional and $G$-irreducible. By \eqref{lemm:ZG:L}, we may find a non-empty, and thus dense, $G$-invariant open immersion $v\colon V \to X/\mun$ and a nowhere vanishing section $l$ of $\Lch|_V$ of weight $g \in \Char{G}$. The image $b \in H^0(V,\Gm)$ of $l^{\otimes n}$ under the morphism $\Lch^{\otimes n} \to \Oc_{X/\mun}$ (induced by the multiplication of $\Oc_X$) has weight $n\cdot g$ as a section of $\Oc_V$, and is a lifting of $\sch|_V \in H^0(V,\Gm/n)$. Thus by \dref{lemm:Gmn}{lemm:Gmn:pushforward}
\[
d \circ u_* \circ \{\sch\}[X/\mun] = d \circ u_* \circ \{\sch\} \circ v_*[V] = d \circ (u \circ v)_* \circ \{\sch|_V\} [V] \in \Zo(\overline{X}/\mun)/n
\]
is the class modulo $n$ of the cycle $d \circ (u \circ v)_* \circ \{b\} [V] \in \Zo(\overline{X}/\mun)$, which belongs to $\Zo_G(\overline{X}/\mun)$ by \rref{lemm:div_weight}.
\end{proof}

\begin{proposition}
\label{prop:deg_Ginv}
Let $G$ be a finite diagonalisable $p$-group containing $\mup$. Let $X$ be a variety with a $G$-action such that $X^\mup$ is projective and $X^G=\varnothing$. Then $\deg \varrho(X) =0 \in \Fp$ (see \rref{def:varrho}).
\end{proposition}
\begin{proof}
By \dref{lemm:ZG}{lemm:ZG:c1}, the cycle class $c(-\Lch)[X^\circ/\mup] \in \CH(X^\circ/\mup)$ is represented by some cycle $\alpha \in \Zo_G(X^\circ/\mup)$. Therefore the modulo $p$ cycle class
\[
\varrho(X) = \partial_X \circ \{\sch\} \circ c(-\Lch)[X^\circ/\mup] \in \CH(X^\mup)/p
\]
is represented by the element $\beta \in \Zo(X^\mup)/p$ such that
\[
i_*(\beta) = d \circ u_* \circ \{\sch\}(\alpha) \in \Zo(X/\mup)/p
\]
where $u\colon X^\circ/\mup \to X/\mup$ is the open immersion, and $i\colon X^\mup \to X/\mup$ the closed immersion. It follows from \dref{lemm:ZG}{lemm:ZG:s} that $i_*(\beta) \in \Zo(X/\mup)/p$ is the image of a cycle $\gamma \in \Zo_G(X/\mup)$. Thus $\beta = i_*(\beta)|_{X^\mup} \in \Zo(X^\mup)/p$ is the image of $\gamma|_{X^\mup} \in \Zo(X^\mup)$, and to conclude the proof it will suffice to show that $\deg (\gamma|_{X^\mup}) = \deg \circ i_*( \gamma|_{X^\mup})$ is divisible by $p$. To do so, it suffices by \rref{lemm:proj_Ginv} to prove that $\deg [Y] \in p \Zz$ when $Y$ is a zero-dimensional $G$-invariant closed subscheme of $X/\mup$ supported on $X^\mup$. By \rref{lemm:deg_nofix} below it will suffice to prove that $Y^G =\varnothing$ for such $Y$. But for any field extension $L/k$, the subset $Y^G(L) \subset (X/\mup)(L)$ is contained in 
\[
((X/\mup)^G)(L) \cap (X^\mup)(L)= ((X^\mup)^G)(L) = X^G(L) = \varnothing.\qedhere
\]
\end{proof}

\begin{lemma}
\label{lemm:deg_nofix}
Let $G$ be a finite diagonalisable $p$-group acting on a zero-dimensional variety $X$. If $X^G=\varnothing$, then $\deg [X]$ is divisible by $p$.
\end{lemma}
\begin{proof}
The variety $X$ is the spectrum of a $\Char{G}$-graded $k$-algebra $A$ such that $\dim_k A < \infty$. Since $X^G = \varnothing$, the variety $X$ is covered by the $G$-invariant open (and closed) subschemes $D(f)$ for $f \in A_g$ with $g \in \Char{G}-\{0\}$. Replacing $X$ with this cover, we may assume that for some $g \in \Char{G}-\{0\}$ the subset $A_g$ contains an element $a \in A^\times$ . Let $C$ be the subgroup of $\Char{G}$ generated by $g$. For any $h\in \Char{G}$, multiplication with $a$ induces an isomorphism of $k$-vector spaces $A_h \to A_{g+h}$, hence $\dim_k A_h$ depends only on the class of $h$ in $\Char{G}/C$. Since the order of $C$ is divisible by $p$, so is for any $E \in \Char{G}/C$ the dimension of the $k$-vector space $\bigoplus_{e \in E} A_e$. Since $A$ is the direct sum of those vector spaces, it follows that $\dim_k A$ is divisible by $p$.
\end{proof}

\section{Isolated fixed points}
\numberwithin{theorem}{section}

\label{sect:mainth}

\begin{theorem}
\label{th:odd}
Let $X$ be a projective variety without connected component of dimension zero over an algebraically closed field. Assume that $\mu_2$ acts on $X$. Then the set underlying $X^{\mu_2}$ cannot consist in an odd number of regular points.
\end{theorem}
\begin{proof}
Assume that $X^{\mu_2} = \{x_1,\cdots,x_n\}$ for some regular points $x_1,\cdots,x_n$ of $X$. 

Let $i\in \{1,\cdots,n\}$, and consider the $\mu_2$-invariant open subscheme $U_i = X - \{x_j|j\neq i\}$ of $X$. Since $\varrho(U_i) \neq 0 \in \CH((U_i)^{\mu_2})/2$ by \rref{prop:nonvanishing} and the set underlying $(U_i)^{\mu_2}$ is the single point $x_i$, it follows that $\deg \varrho(U_i) \neq 0 \in \mathbb{F}_2$. Therefore $\deg \varrho(U_i) =1 \in \mathbb{F}_2$. 

Now the open subschemes $U_i$ cover $X$, so that repeated applications of \rref{prop:rho_open} yield
\[
\deg \varrho(X) = \deg  \varrho(U_1) + \cdots + \deg  \varrho(U_n) = n \mod 2\in \mathbb{F}_2.
\]
It follows from the assumptions on $X$ and \rref{cor:df} that $\deg \varrho(X) = 0 \in \mathbb{F}_2$, and we conclude that $n$ must be even.
\end{proof}

\begin{theorem}
\label{th:isolated}
Let $X$ be a projective variety without connected component of dimension zero, and $G$ a finite diagonalisable $p$-group acting on $X$. Then the set underlying $X^G$ cannot be a single regular closed point of $X$ of degree prime to $p$.
\end{theorem}
\begin{proof}
We proceed by induction on $|G|$. If $G$ is trivial, then $X^G =X$ has no isolated point, and the statement is proved. When $G$ is non-trivial, it contains $\mup$ as a subgroup (because the ordinary $p$-group $\Char{G}$ admits $\Zz/p$ as a quotient). We assume that the set underlying the closed subscheme $X^G$ is a single regular point $x$ of $X$ of degree prime to $p$. Let $Y$ be the connected component of $X^\mup$ containing $X^G$, and $Y'$ the maximal $G$-invariant open subscheme of $X^\mup$ contained in $Y$. Then $Y'$ is closed in $X^\mup$ by \tref{prop:open_action}{prop:open_action:sat}{prop:open_action:sat:closed}, and non-empty because $x \in Y'$ by \tref{prop:open_action}{prop:open_action:sat}{prop:open_action:sat:inv}. Since $Y$ is connected, it follows that $Y'=Y$. Thus $Y$ is a $G$-invariant open subscheme of $X^\mup$, and so is $Z=X^\mup-Y$. The point $x$ is contained in some regular open subscheme $R$ of $X$, which may be assumed to be $G$-invariant by \tref{prop:open_action}{prop:open_action:sat}{prop:open_action:sat:inv}. It follows from \rref{lemm:fixed_reg} that the variety $R^\mup$, hence also its open subscheme $R^\mup \cap Y$ is regular. Thus $x$ (being contained in the open subscheme $R^\mup \cap Y$) is a regular closed point of $Y$. The induction hypothesis applied to the action of the diagonalisable group $G/\mup = \Diag(\ker (\Char{G} \to \Char{\mup}))$ on $Y$ shows that $\dim Y=0$. Since $Y$ is connected and regular at $x$, it coincides with the closed point $x$ with reduced structure. As $X^G$ is a closed subscheme of $Y$ containing $x$, we must have $X^G=Y$. Thus we have obtained a $G$-equivariant decomposition as disjoint open subschemes $X^\mup = X^G \sqcup Z$.

The fixed loci $(X-Z)^\mup=X^G$ and $(X-X^G)^\mup=Z$ are projective. Since $X$ is projective, we have by \rref{cor:df} and \rref{prop:rho_open}
\[
0 = \deg \varrho(X) = \deg \varrho(X-Z) + \deg \varrho(X-X^G) \in \Fp.
\]
Since $\varrho(X-Z) \neq 0 \in \CH(X^G)/p$ by \rref{prop:nonvanishing} and $\deg[X^G]=[k(x):k] \not \in p\Zz$, we deduce that $\deg \varrho(X-Z) \neq 0 \in \Fp$. Thus $\deg \varrho(X-X^G) \neq 0 \in \Fp$, contradicting \rref{prop:deg_Ginv}.
\end{proof}

\end{document}